\documentclass[review,3p,square,sorted]{elsarticle}

\usepackage{amssymb}
\usepackage{amsmath}
\usepackage{amsthm}
\usepackage{stmaryrd}
\usepackage{bbm}
\usepackage{enumitem}

\usepackage{fancyhdr}
\usepackage{hyperref}
\hypersetup{%
  pdftitle={BSDEs},%
  pdfsubject={BSDEs},%
  pdfauthor={ShengJun FAN, Ying Hu},%
  pdfkeywords={BSDEs},%
  pdfstartview=FitH,%
  CJKbookmarks=true,%
  bookmarksnumbered=true,%
  bookmarksopen=true,%
  colorlinks=true, linkcolor=blue, urlcolor=blue, citecolor=blue, %
}

\usepackage{cleveref}
\crefname{thm}{Theorem}{Theorems}
\crefname{pro}{Proposition}{Propositions}
\crefname{lem}{Lemma}{Lemmas}
\crefname{rmk}{Remark}{Remarks}
\crefname{cor}{Corollary}{Corollaries}
\crefname{dfn}{Definition}{Definitions}
\crefname{ex}{Example}{Examples}
\crefname{section}{Section}{Sections}
\crefname{subsection}{Subsection}{Subsections}

 \newcommand{\as}{{\rm d}\mathbb{P}\times {\rm d}t-a.e.}
 
 \newcommand{\R}{\mathbb{R}}
 
 \newcommand{\E}{\mathbb{E}}
 \newcommand{\F}{\mathcal{F}}
 
 \newcommand{\s}{\mathcal{S}}
 \newcommand{\M}{\mathcal{M}}
 \newcommand{\dd}{\mathrm{d}}
 
 \newcommand {\Dis}{\displaystyle}
\newcommand{\ass}{{\rm d}\mathbb{P}-a.s.}

\newcommand{\T}{[0,T]}

\newcommand{\dif}{{\rm d}}

\newcommand{\limn}{\lim\limits_{n\rightarrow\infty}}

\newcommand{\supt}{\sup\limits_{t\in[0,T]}}
\newcommand{\supn}{\sup\limits_{n\geq1}}

\newcommand{\supi}{\sup\limits_{i\geq1}}

\newtheorem{thm}{Theorem}[section]

\newtheorem{pro}[thm]{Proposition}
\newtheorem{rmk}[thm]{Remark}

\newtheorem{dfn}[thm]{Definition}
\newtheorem{ex}[thm]{Example}

\renewcommand{\theequation}{\arabic{section}.\arabic{equation}}

\journal{arXiv}

\begin{document}
\begin{frontmatter}
\title{Existence and uniqueness on $L^1$ solutions of multidimensional BSDEs with generators of stochastic one-sided Osgood type\tnoteref{found}}
\tnotetext[found]{This work is partially funded by the National Natural Science Foundation of China (No.12171471).
\vspace{0.2cm}}

\author[Fan,Lai]{Yuru Lai} \ead{1515388061@qq.com}
\author[Fan]{Xinying Li} \ead{lixinyingcumt@163.com}
\author[Fan]{Shengjun Fan}\ead{shengjunfan@cumt.edu.cn, Corresponding author} \vspace{-0.5cm}

\affiliation[Fan]{organization={School of Mathematics, China University of Mining and Technology},
            city={Xuzhou 221116},
            country={China}}

\affiliation[Lai]{organization={Henan Key Laboratory of Network Cryptography Technology},
            city={Zhengzhou 450000},
            country={China}}


\begin{abstract}
This paper establishes an existence and uniqueness result for $L^1$ solutions of multidimensional backward stochastic differential equations (BSDEs) with a general terminal time when the generator $g$ satisfies a stochastic one-sided Osgood condition along with a general growth condition in the state variable $y$ and a stochastic Lipschitz condition in the state variable $z$, extending and strengthening Theorems 1 and 2 of Fan [J. Theor. Probab. 31(2018)]. In particular, it is explored that the first component of the $L^1$ solution satisfies a new integrability condition. Two general stochastic Gronwall-type and Bihari-type inequalities along with some innovative techniques dealing with stochastic coefficients and weaker integrability conditions play crucial roles in our proofs, and can be useful in further study on adapted solutions of BSDEs.  \vspace{0.2cm}
\end{abstract}

\begin{keyword}
Backward stochastic differential equation \sep Stochastic one-sided Osgood condition \sep \\
\hspace*{1.8cm} Stochastic Lipschitz condition \sep $L^1$ solution \sep Existence and uniqueness. \vspace{0.2cm}
\end{keyword}

\end{frontmatter}
\vspace{-0.4cm}

\section{Introduction}
\label{sec:1-Introduction}
\setcounter{equation}{0}

Let $k$ and $d$ be two positive integers, $(\Omega,\F,\mathbb{P})$ be a probability space carrying a $d$-dimensional standard Brownian motion $(B_t)_{t\ge0}$, and $(\F_t)_{t\ge0}$ be the natural $\sigma$-algebra flow generated by $(B_t)_{t\ge0}$. Assume that $0<T\leq +\infty$ is a given time horizon and $\F:=\F_T$. All equalities and inequalities between random elements will be understood in $\ass$ sense. We consider multidimensional backward stochastic differential equations (BSDEs) of the following form:
\begin{align}\label{101}
     y_t=\xi+\int_t^Tg(s,y_s,z_s) {\rm d}s-\int_t^Tz_s {\rm d}B_s, \ \ \ t\in[0,T],
\end{align}
where $T$ is called the terminal time, $\xi$ is called the terminal condition which is an $\R^k$-valued random vector of $\F_T$-measurable, the random function $$g(\omega,t,y,z):\Omega\times[0,T]\times\R^k\times\R^{k\times d}\mapsto\R^k$$
is $(\F_t)$-progressively measurable for each $(y,z)\in \R^k\times\R^{k\times d}$ which is called the generator.\vspace{0.2cm}

It is well known that \cite{PardouxPeng1990} first proposed the nonlinear BSDE (\ref{101}) with a finite terminal time and established the existence and uniqueness of the $L^2$ solution when the generator $g$ satisfies a uniform Lipschitz condition in the state variables $(y,z)$, and both $\xi$ and $g(t,0,0)$ satisfy a square-integrable condition. After that, BSDEs have attracted a lot of research interest and become an important mathematical tool in many fields such as partial differential equations, nonlinear mathematical expectation, financial mathematics, stochastic optimal control and stochastic games. See, for example, \cite{Peng1997,El97,Kob00,Jia10,Delb10,Hu11,HuTang2016SPA,FanHuTang2023SCL,FanHuTang2024SCL2} for more details. At the same time, numerous works have been carried out on the existence and uniqueness of adapted solutions of BSDEs under weaker assumptions to meet the requirements of various practical applications. For instance, in \cite{ChenWang00,Fan16,ZongHu18,XiaoFan20,LiXuFan2021,LiLaiFan2023,LiFan2023} the authors were devoted to extending the finite terminal time to the general one (i.e., $T$ may be positive infinity); in \cite{Bri03,Fan15,LiuLiFan2020,LiFan2023,FanHuTang2023SCL,FanHuTang2024SCL2,Klimsiak2024ECP} the authors aimed to relax the square-integrable condition on $\xi$ and $g(t,0,0)$; and in \cite{Mao1995,Bri03,Ham03,BriHu06,BriHu08,FanJiang10,FanJiang13,Jia10,
HuTang2016SPA,LiXuFan2021,LiLaiFan2023,LiFan2023,FanHuTang2023JDE} the authors were interested in weakening the uniform Lipschitz hypothesis on the generator $g$.\vspace{0.2cm}

With the $L^2$ solution of BSDE \eqref{101}, \citet{Peng1997} introduced the notion of $g$-expectation, a kind of nonlinear expectations, for square-integrable random variables. Since the conventional expectation is defined in the $L^1$ space, the $L^1$ solution of BSDE \eqref{101} should be investigated in order to introduce the $g$-expectation of only integrable random variables. This paper aims to study the existence and uniqueness of the $L^1$ solution of multidimensional BSDE \eqref{101} with a general terminal time under some weaker assumptions on the generator $g$. It should be noted that the $L^p$ solution is more difficult to handle when $p\in [1,+\infty)$ becomes smaller since the smoothness of the function $f(x)=|x|^p,\, x\in\R^k$ becomes weaker. Generally speaking, for the $L^p$ solution, it is the case of $p=2$ that was first investigated, then the case of $p>2$ and the case of $p\in (1,2)$, and finally the case of $p=1$.\vspace{0.2cm}

In the sequel, let us briefly introduce some known results on the $L^p~(p\geq1)$ solution of BSDE (\ref{101}), especially the $L^1$ solution. When the terminal time $T$ is a finite positive constant and the generator $g$ satisfies a monotonicity condition along with a general growth condition in the state variable $y$ and a Lipschitz condition in the state variable $z$, \citet{Bri03} studied the $L^p~(p\geq 1)$ solution of BSDE (\ref{101}), and established the existence and uniqueness of the $L^1$ solution under an additional assumption that $g$ is also of sub-linear growth in $z$. This work was wholly extended in \citet{XiaoFanXu15} to BSDE (\ref{101}) with a general terminal time where those constant coefficients appearing in the conditions of the generator $g$ with respect to the state variables $(y,z)$ are replaced with some time-varying ones satisfying certain integrability conditions. \citet{Fan18} further generalized the result of the $L^1$ solution in \cite{Bri03} by weakening the monotonicity condition of $g$ in $y$ to a one-sided Osgood condition. Recently, by establishing a stochastic Gronwall-type inequality and a stochastic Bihari-type inequality via the martingale representation theorem and the BMO martingale method, the authors in \citet{LiXuFan2021} and \citet{LiFan2023}, supposing that those coefficients appearing in the conditions of $g$ in $(y,z)$ may depend on $t$ and $\omega$ and satisfy certain integrability conditions, obtained the existence and uniqueness of the $L^p~(p>1)$ solution of BSDE (\ref{101}) with a general terminal time $T$ and a generator $g$ satisfying a $p$-order weak stochastic monotonicity condition along with a general growth condition in the state variable $y$ (see assumption (H1a)$_p$ in \cref{sec:4-The proof of main result} and assumption (H2) in \cref{sec:3-The statement of Main result}) and a stochastic Lipschitz condition in the state variable $z$ (see assumption (H3) in \cref{sec:3-The statement of Main result}), extending further the corresponding results obtained in \cite{XiaoFanXu15}. Then, a question is naturally asked: under an additional stochastic sub-linear growth condition of $g$ in $z$ similar to that in \cite{Bri03,Fan18} (see assumption (H4) in \cref{sec:3-The statement of Main result}), can we establish a corresponding result to those in \cite{LiXuFan2021,LiFan2023} for the case of $p=1$? The present paper will give an affirmative answer of this question.
Note that the $p$-order weak stochastic monotonicity condition (H1a)$_p$ is just the stochastic one-sided Osgood condition (see assumption (H1) in \cref{sec:3-The statement of Main result}) when $p=1$. By the way, we also mention that very recently, \citet{FanHuTang2023SCL,FanHuTang2024SCL2} studied the $L^1$ solution of a one-dimensional BSDE when the generator $g$ admits a general iterated-logarithmically sub-linear growth in the state variable $z$, strengthening the corresponding result of \cite{Bri03} in the one-dimensional setting.\vspace{0.2cm}

More specifically, this paper establishes a general existence and uniqueness result on the $L^1$ solution of BSDE (\ref{101}) when the generator $g$ satisfies the aforementioned weaker assumptions (H1)-(H4) (see \cref{th} in \cref{sec:3-The statement of Main result}), extending and strengthening Theorems 1 and 2 in \citet{Fan18} (see \cref{rm3.3} in \cref{sec:3-The statement of Main result}). In particular, it is explored that the first component of the $L^1$ solution satisfies a new integrability condition. Some core difficulties arise naturally due to the combination of stochastic coefficients, $L^1$ integrability and one-sided Osgood condition. For instance, the stochastic Gronwall-type and Bihari-type inequalities used in \cite{LiXuFan2021,LiFan2023} for the case of the $L^p~(p>1)$ solution is not invalid for the case of the $L^1$ solution, and a key a priori estimate on the $L^p~(p>1)$ solution is not established under a new condition related to the stochastic one-sided Osgood condition. In particular, when $(u_t)_{t\in\T}$ is a nonnegative deterministic process and $\rho(\cdot)$ is a concave function, by Fubini's theorem and Jensen's inequality we have for each integrable process $(y_t)_{t\in\T}$,
$$
\E\left[\int_0^T u_t\rho(|y_t|)\dd t\right]\leq \int_0^T u_t\rho(\E[|y_t|]) \dd t.
$$
However, the last assertion is not correct any longer when $u_\cdot$ is not a deterministic process. To overcome these difficulties, we extend the stochastic Gronwall-type and Bihari-type inequalities obtained in \citet{LiXuFan2021} (see \cref{pro:2.3,pro:2.4} in \cref{sec:2-Preliminaries}), establish a new a priori estimate on the $L^p$ solution of BSDE \eqref{101} (see \cref{pro:4.7} in \cref{sec:4-The proof of main result}), and develop innovative techniques to apply those stochastic inequalities and a priori estimates.\vspace{0.2cm}

The rest of this paper is organized as follows. \cref{sec:2-Preliminaries} contains some notations, spaces and definitions as well as two stochastic inequalities. In \cref{sec:3-The statement of Main result} we present our main result--\cref{th} and provide two examples as its application. The detailed proof of \cref{th} is provided in \cref{sec:4-The proof of main result}, and the proof of \cref{pro:4.7} is deferred to the Appendix.

\section{Preliminaries\vspace{0.2cm}}
\label{sec:2-Preliminaries}
\setcounter{equation}{0}

\subsection{Notations and definitions\vspace{0.1cm}}

Let us first introduce some notations to be used. Let $\R_+:=[0,+\infty)$ and $\mathbbm{1}_{A}$ denote the indicator function of a set $A$. For each $x,y\in\R$, let $x\vee y:=\max\{x,y\}, x\wedge y:=\min\{x,y\}$. let $|y|$ and $|z|$ denote the Euclidean norm of vector $y\in \R^k$ and matrix $z\in \R^{k\times d}$, respectively, and let $\langle x,y \rangle$ denote the usual inner product of $x,y\in\R^k$.
For each $p>0$, by $L^p(\F_T;\R^k)$ and $L^\infty(\F_T;\R^k)$ we denote the set of $\F_T$-measurable $\R^k$-valued random vectors $\xi$ such that $\E\left [|\xi|^p\right ]<+\infty$ and $\|\xi\|_\infty:=\sup\{x: \mathbb{P}(|\xi|>x)>0\}<+\infty$, respectively. Furthermore, for $p>0$, by $\mathcal{H}^p(0,T;\R)$ ($\mathcal{H}^p$ for short) we denote the set of $(\F_t)$-progressively measurable $\R$-valued processes $(f_t)_{t\in[0,T]}$ such that
$$\|f\|_{\mathcal{H}^p}:=\left(\E\left[\left(\int_0^T |f_t| \dd t\right)^p\right]\right)^{\frac{1}{p}\wedge1},$$
by $\s^p(0,T;\R^k)$ ($\s^p$ for short) the set of $(\F_t)$-adapted, $\R^k$-valued and continuous processes $(Y_t)_{t\in[0,T]}$ such that
$$ \|Y\|_{\s^p}:=\left(\E\left[\sup\limits_{t\in[0,T]}|Y_t|^p\right]
\right)^{\frac{1}{p}\wedge1}<+\infty,\vspace{0.1cm}$$
and by $\M^p(0,T;\R^{k\times d})$ ($\M^p$ for short) the set of $(\F_t)$-progressively measurable $\R^{k\times d}$-valued processes $(Z_t)_{t\in[0,T]}$ such that
$$\|Z\|_{\M^p}:=\left(\E\left[\left(\int_0^T |Z_t|^2{\rm d}t\right)^{\frac{p}{2}}\right]\right)^{\frac{1}{p}\wedge1}
<+\infty.\vspace{0.1cm}
$$
For $p\geq1$, $\s^p$ and $\M^p$ are both Banach spaces, and for each $p\in (0,1)$, they are both complete metric spaces. We recall that a process $(Y_t)_{t\in[0,T]}$ belongs to class $(D)$, if the family of variables
\{$|Y_{\tau}|:\tau$ is the $(\F_t)$-measurable stopping time valued on $[0,T]$ \} is uniformly integrable. The space of $(\F_t)$-progressively measurable continuous processes $(Y_t)_{t\in[0,T]}$ belonging to class $(D)$ is complete under the norm:
$$\left\|Y\right\|_1:=\supt \E\left[|Y_t|\right].$$
For each $p>0$, by $L^\infty(\Omega;L^p([0,T];\R_+))$ we denote the set of $(\F_t)$-progressively measurable processes $f_t(\omega):\Omega\times[0,T]\mapsto\R_+$ satisfying
$$\bigg\|\int_0^T f_t^p \dd t \bigg\|_\infty < +\infty.$$
Finally, by ${\bf S}$ we denote the set of nondecreasing, concave and differentiable functions $\rho(x):\R_+\rightarrow\R_+$ satisfying the following conditions:\vspace{0.2cm}

(i) $\rho(0)=0$, $\rho(u)>0$ for each $u>0$ and $\int_{0^+}\frac{\dd u}{\rho(u)} = +\infty$;\vspace{0.1cm}

(ii) $\rho'(u)$ is bounded on $[c,\bar{c}]$ for any $0<c<\bar{c}<+\infty$.\vspace{0.3cm}

Let us further present the following two definitions related to the adapted solution of BSDE (\ref{101}).\vspace{-0.2cm}

\begin{dfn}
A solution of BSDE (\ref{101}) is a pair of $(\F_t)$-progressively measurable processes $(y_t,z_t)_{t\in[0,T]}$ with values in ${\R}^k\times {\R}^{k\times d}$, such that $\ass$, $\int_0^T|z_t|^2\ {\rm d}t<+\infty$, $\int_0^T|g(t,y_t,z_t)|\ {\rm d}t<+\infty$, and BSDE (\ref{101}) is fulfilled for each $t\in[0,T]$.
\end{dfn}

\begin{dfn}{\label{de2}}
Suppose that $(y_t,z_t)_{t\in[0,T]}$ is a solution of BSDE (\ref{101}). If for some $p>1$, $(y_t,z_t)_{t\in[0,T]}\in \s^p\times \M^p$, then it is called an $L^p$ solution of BSDE (\ref{101}). If $(y_t,z_t)_{t\in[0,T]}\in \bigcap_{\beta\in(0,1)}\s^\beta\times \M^\beta$ and $(y_t)_{t\in[0,T]}$ belongs to class (D), then it is called an $L^1$ solution of BSDE (\ref{101}).
\end{dfn}

\subsection{Two stochastic inequalities\vspace{0.1cm}}

In this subsection, we are going to present two stochastic inequalities extending slightly Propositions 3.1 and 3.2 of \citet{LiXuFan2021}, where the processes $\bar{\mu}_\cdot$ and $\tilde{\mu}_\cdot$ satisfy some stronger integrability conditions. They will play crucial roles in the proof of the main result of this paper.

\begin{pro}[Stochastic Gronwall-type inequality]\label{pro:2.3}
Assume that $\eta$ is an $\F_T$-measurable nonnegative random variable satisfying $\E[\eta] < +\infty$, $\beta_\cdot\in L^{\infty}(\Omega ; L^{1}([0, T]; \R_+))$, $\bar {f}_\cdot$ and $\bar {l}_\cdot$ are two $(\F_t)$-progressively measurable nonnegative processes belonging to $\mathcal{H}^1(0,T;\R)$, and $\bar {\mu}_\cdot$ is an $(\F_t)$-progressively measurable, continuous and nonnegative process satisfying
\begin{equation}\label{eq:2.1}
\mathbb{E}\left[\int_0^T \beta_s\bar{\mu}_{s} \dd s\right]<+\infty.
\end{equation}
If
$$
\bar{\mu}_{t} \leq \mathbb{E}\left[\eta+\int_{t}^{T}\left(\beta_{s} \bar{\mu}_{s}+\bar{f}_{s}\right) \mathrm{d}s \bigg| \mathcal{F}_{t}\right], \ \ \ \ t\in[0,T],
$$
then
$$
 \bar{\mu}_{t} \leq e^{\left\|\int_{t}^{T} \beta_{s} \mathrm{d}s \right\|_{\infty}} \E\left[\eta+\int_{t}^{T} \bar{f}_{s} \mathrm{d} s \bigg| \mathcal{F}_{t}\right], \ \ \ \ t\in[0,T].
$$
Moreover, if
$$
\mathbb{E}\left[\sup _{s \in[t, T]} \bar{\mu}_{s} + \int_t^T \bar{l}_s \dd s \bigg| \mathcal{F}_{t}\right] \leq \mathbb{E}\left[\eta+\int_{t}^{T}\left(\beta_{s} \bar{\mu}_{s}+\bar{f}_{s}\right) \mathrm{d}s \bigg| \mathcal{F}_{t}\right], \ \ \ \ t\in[0,T],
$$
then
$$
\bar{\mu}_{t} \leq \mathbb{E}\left[\sup _{s \in[t, T]} \bar{\mu}_{s} + \int_t^T\bar{l}_s \dd s \bigg| \mathcal{F}_{t}\right] \leq e^{\left\|\int_{t}^{T} \beta_{s} \mathrm{d}s\right\|_{\infty}} \mathbb{E}\left[\eta+\int_{t}^{T} \bar{f}_{s} \mathrm{d}s \bigg| \mathcal{F}_{t}\right], \ \ \ \ t\in[0,T].
$$
\end{pro}

\begin{proof}
Set
$$
\bar{\eta}:= \eta+\int_{0}^{T}\left(\beta_{s} \bar{\mu}_{s}+\bar{f}_{s}\right) \mathrm{d}s.
$$
By the assumptions of $\eta$, $\beta_\cdot$, $\bar\mu_\cdot$ and $\bar f_\cdot$, we can conclude that $\mathbb{E}[\bar{\eta}]<+\infty$. The rest of proof runs as that in the proof of Proposition 3.1 in \citet{LiXuFan2021}.
\end{proof}

\begin{pro}[Stochastic Bihari-type inequality]\label{pro:2.4}
Assume that $c>0$, $\beta_\cdot \in L^{\infty}(\Omega ; L^{1}([0, T]; \R_+))$, $\rho(\cdot)\in{\bf S}$, and $\widetilde{\mu_\cdot}$ is an $(\F_t)$-progressively measurable, continuous and nonnegative process satisfying
\begin{equation}\label{eq:2.2}
\mathbb{E}\left[\int_0^T\beta_t\tilde{\mu}_{t}\dd t\right]<+\infty.
\end{equation}
If
$$
\tilde{\mu}_{t} \leq c+\mathbb{E}\left[\int_{t}^{T} \beta_{s} \rho\left(\tilde{\mu}_{s}\right) \mathrm{d} s \bigg| \mathcal{F}_{t}\right], \ \ \ \ t\in[0,T],
$$
then
$$
\tilde{\mu}_{t} \leq \Theta^{-1}\left(\Theta(c)+\left\|\int_{t}^{T} \beta_{s} \mathrm{ d} s\right\|_{\infty}\right), \ \ \ \ t\in[0,T],
$$
where
$$
\Theta(x):=\int_{1}^{x} \frac{1}{\rho(u)} \mathrm{d} u, \ \ \ \ x>0\vspace{0.2cm}
$$
is a strictly increasing function valued in $\R$, and $\Theta^{-1}(\cdot)$ is the inverse function of $\Theta(\cdot)$.

Moreover, if $c=0$, then for each $t\in[0,T]$, $\tilde{\mu}_{t} \equiv 0$.
\end{pro}

\begin{proof}
Since $\rho(x)\leq A(1+x)$ for some constant $A>0$ (see \eqref{eq:4.1} in \cref{sec:4-The proof of main result} for details), by \cref{pro:2.3} we deduce that
$$
0\leq \tilde{\mu}_t\leq e^{A\|\int_0^T\beta_s\dd s\|_\infty}\left(c+A\left\|\int_0^T\beta_s\dd s \right\|_{\infty}\right), \ \ \ \ t\in[0,T].\vspace{0.1cm}
$$
Thus, in light of the assumptions of $\rho(\cdot)$, the rest of proof runs as that in the proof of Proposition 3.2 in \citet{LiXuFan2021}.\vspace{0.2cm}
\end{proof}

\begin{rmk} We have the following remarks.
\begin{itemize}
\item [(i)] Note that the processes $\bar{\mu}_\cdot$ and $\tilde{\mu}_\cdot$ in Propositions 3.1 and 3.2 of \citet{LiXuFan2021} need to satisfy, respectively,
\begin{equation}\label{eq:2.3}
\E\left[\sup_{t\in\T}\bar\mu_t\right]<+\infty\ \ \ {\rm and}\ \ \ \E\left[\sup_{t\in\T}\tilde\mu_t\right]<+\infty.
\end{equation}
It is evident that these two conditions are strictly stronger than \eqref{eq:2.1} and \eqref{eq:2.2}, respectively, due to the fact of $\beta_\cdot\in L^{\infty}(\Omega ; L^{1}([0, T]; \R_+))$. Consequently, Propositions \ref{pro:2.3} and \ref{pro:2.4} improves the corresponding results obtained in \cite{LiXuFan2021}.

\item [(ii)] We would like to emphasize that due to the $L^1$ solution being focused on in this paper, the first component of the solution of a BSDE usually does not belong to the space $\s^1$, which means that the conditions appearing in \eqref{eq:2.3} can not be usually satisfied, and then Propositions 3.1 and 3.2 of \cite{LiXuFan2021} are invalid. However, it will be seen that the conditions \eqref{eq:2.1} and \eqref{eq:2.2} can be satisfied in our study on the $L^1$ solution, and Propositions \ref{pro:2.3} and \ref{pro:2.4} are sufficient to address our problem.

\item [(iii)] It seems to be true that \eqref{eq:2.1} and \eqref{eq:2.2} are the most elementary conditions ensuring that \cref{pro:2.3,pro:2.4} hold.
\end{itemize}
\end{rmk}

\section{Main result\vspace{0.2cm}}
\label{sec:3-The statement of Main result}
\setcounter{equation}{0}

\subsection{Existence and uniqueness\vspace{0.1cm}}

We first introduce the following assumptions on the generator $g$ and the terminal condition $\xi$, where the terminal time satisfies $0< T \leq +\infty$.

\begin{enumerate}
\renewcommand{\theenumi}{(H1)}
\renewcommand{\labelenumi}{\theenumi}
\item  $g$ satisfies a stochastic one-sided Osgood condition in $y$, i.e., there exist a process $u_\cdot\in L^\infty(\Omega;$ $L^1([0,T];\R_+))$ and a function $\rho(\cdot)\in {\bf S}$ such that $\as$, for each $y_1,y_2\in \R^k,z\in\R^{k\times d}$,
    \begin{equation*}
        \left\langle {y_1-y_2\over |y_1-y_2|}\mathbbm{1}_{|y_1-y_2|\neq 0},\ \ g(\omega,t,y_1,z)-g(\omega,t,y_2,z)\right\rangle\leq u_{t}(\omega) \rho(|y_1-y_2|).
    \end{equation*}
\end{enumerate}
\begin{enumerate}
\renewcommand{\theenumi}{(H2)}
\renewcommand{\labelenumi}{\theenumi}
\item  $g$ has a general growth in $y$, i.e., for each $r\in\R_+$, it holds that
    \begin{equation*}
        \E\left[\int_0^T \psi_r(\omega,t)\ {\rm d}t\right]<+\infty
    \end{equation*}
with
    \begin{equation*}
    \psi_r(\omega,t):=\sup\limits_{|y|\leq r}|g(\omega,t,y,0)|.
    \end{equation*}
Moreover, $\as$, $g(\omega,t,\cdot,z)$ is continuous for each $z\in \R^{k \times d}$.
\end{enumerate}
\begin{enumerate}
\renewcommand{\theenumi}{(H3)}
\renewcommand{\labelenumi}{\theenumi}
\item  $g$ satisfies a stochastic Lipschitz condition in $z$, i.e., there exists a process $v_\cdot\in L^\infty(\Omega;L^2([0,T];$ \\ $\R_+))$ such that $\as$, for each $y\in \R^k, z_1,z_2\in \R^{k\times d}$,
    \begin{equation*}
        |g(\omega,t,y,z_1)-g(\omega,t,y,z_2)|\leq v_t(\omega)|z_1-z_2|.
    \end{equation*}
\end{enumerate}
\begin{enumerate}
\renewcommand{\theenumi}{(H4)}
\renewcommand{\labelenumi}{\theenumi}
\item (i) $g$ satisfies a stochastic sub-linear growth condition in $z$, i.e., there exist a constant $\alpha \in[0,1)$ and an $(\F_t)$-progressively measurable nonnegative process $(\gamma_t)_{t\in\T}$ satisfying
    \begin{equation}\label{eq:3.1}
         \bigg\|\int_0^T \left(\gamma_t(\omega)+\gamma_t^{\frac{1}{1-\alpha}}(\omega)\right) \dif t \bigg\|_\infty
         <+\infty
    \end{equation}
such that $\as$, for each $y\in\R^k, z\in\R^{k\times d}$,
    \begin{equation*}
         |g(\omega,t,y,z)-g(\omega,t,y,0)|\leq \gamma_t(\omega) (g^1_t(\omega)+g^2_t(\omega)+|y|+|z|)^\alpha,
    \end{equation*}
where $(g^1_t)_{t\in[0,T]}$ and $(g^2_t)_{t\in[0,T]}$ are $(\F_t)$-progressively measurable and nonnegative processes satisfying
\begin{equation}\label{eq:3.1*}
\E\left[\int_0^T |g^1_t|\dd t\right]<+\infty\ \ \  {\rm and}\ \ \ \E\left[\sup_{t\in [0,T]} |g^2_t|\right]<+\infty.\vspace{0.3cm}
\end{equation}

(ii) When the $\alpha$ defined in (i) takes values on [1/2,1), we further assume that there exists a constant $\bar p>1$ such that the function $\rho(\cdot)$ in (H1) satisfies
    \begin{equation}\label{eq:3.2}
    \int_{0^+}\frac{u^{\bar p-1}}{\rho^{\bar p}(u)}\dif u=+\infty.
    \end{equation}
\end{enumerate}
\begin{enumerate}
\renewcommand{\theenumi}{(H5)}
\renewcommand{\labelenumi}{\theenumi}
\item $\E\left[|\xi|+\int_0^T|g(t,0,0)| \dd t\right]<+\infty.$\vspace{0.2cm}
\end{enumerate}

\begin{rmk}\label{rm3.1} Concerning the previous assumptions, we make the following comments.
\begin{itemize}
\item [(i)] Since $(u_t)_{t\in[0,T]}, (v_t)_{t\in[0,T]}$ and $(\gamma_t)_{t\in[0,T]}$ in (H1), (H3) and (H4) may depend on $t$ and $\omega$, where $0<T\leq+\infty$, then assumption (H1) is strictly weaker than the uniform one-sided Osgood condition of $g$ in $y$ employed in \citet{Fan18} where the $u_\cdot$ is a constant independent of $t$ and $\omega$, and assumptions (H3) and (H4) are strictly weaker than the uniform Lipschitz condition and the sub-linear growth of $g$ in $z$ used in \citet{Fan18} where $v_\cdot$ and $\gamma_\cdot$ are also constants independent of $t$ and $\omega$, and $g^2_\cdot\equiv 0$.

\item [(ii)] The inequality \eqref{eq:3.1} in (i) of (H4) imply
\begin{equation}\label{eq:3.3}
\bigg\|\int_0^T \gamma_t^{\frac{2}{2-\alpha}}(\omega) \dif t \bigg\|_\infty<+\infty.
\end{equation}
In fact, note that $1<\frac{2}{2-\alpha}<\frac{1}{1-\alpha}$ for $\alpha\in(0,1)$. For each $x\geq0$ and $\alpha\in [0,1)$, we have
$$x^{\frac{2}{2-\alpha}}\leq x+x^{\frac{1}{1-\alpha}},$$
then the desired conclusion \eqref{eq:3.3} follows.

\item [(iii)] Let $\rho(\cdot)\in {\bf S}$. Since $\rho(\cdot)$ is a concave function on $\R_+$ with $\rho(0)=0$, we have $\rho(u)\geq u\rho(1),\ u\in (0,1)$, and then for each $1\leq \bar p\leq \bar q$,
$$
\frac{\frac{u^{\bar p-1}}{\rho^{\bar p}(u)}}{\frac{u^{\bar q-1}}{\rho^{\bar q}(u)}}=\left(\frac{\rho(u)}{u}\right)^{\bar q-\bar p}\geq (\rho(1))^{\bar q-\bar p},\ u\in (0,1).
$$
Thus, if $\int_{0^+}\frac{u^{\bar q-1}}{\rho^{\bar q}(u)}\dif u=+\infty$, then $\int_{0^+}\frac{u^{\bar p-1}}{\rho^{\bar p}(u)}\dif u=+\infty$, which means that the condition \eqref{eq:3.2} becomes stronger as $\bar p$ becomes bigger. In particular, if \eqref{eq:3.2} holds for some $\bar p>1$, then $\int_{0^+}\frac{1}{\rho(u)}\dif u=+\infty$. 

\item [(iv)] It is evident that when $\rho(x)=kx$ for some constant $k>0$,
the inequality \eqref{eq:3.2} in (ii) of (H4) is trivially satisfied for each $\bar p\geq 1$. Furthermore, for any $\rho(\cdot)\in {\bf S}$, it can be easily verified that if there exist three constants $\delta\in (0,1)$, $c>0$ and $\beta\in (0,1]$ such that $\rho(u)\leq cu|\ln u|^{\beta}$ for each $u\in (0,\delta)$, then for the case of $\beta\in (0,1)$, \eqref{eq:3.2} holds for any $\bar p\in (1,1/\beta]$, while for the case of $\beta=1$, we have $\int_{0^+}\frac{1}{\rho(u)}\dif u=+\infty$, however, \eqref{eq:3.2} does not hold for any $\bar p>1$. Therefore, \eqref{eq:3.2} is not a strongly restrictive condition.
\end{itemize}
\end{rmk}

The following existence and uniqueness theorem is the main result of this paper.

\begin{thm}\label{th}
Let assumptions (H1)-(H5) be in force. Then BSDE (\ref{101}) admits a unique $L^1$ solution $(y_t,z_t)_{t\in[0,T]}$ satisfying
\begin{equation}\label{eq:3.4}
	\E\left[\int_0^T u_t|y_t| \dd t\right]<+\infty,
\end{equation}
where $u_\cdot$ is defined in (H1). Moreover, there exist a process $(Y_t)_{t\in\T}\in \s^q(0,T;\R^k)$ for some $q>1$ and a nonnegative random variable $\eta\in L^1(\F_T;\R^k)$ such that for each $t\in\T$, we have
\begin{equation}\label{eq:3.5}
	|y_t|\leq |Y_t|+\E[\eta|\F_t].\vspace{0.2cm}
\end{equation}
\end{thm}

For readers' convenience, here we would like to emphasize the integrability conditions of those processes involved in Thoerem \ref{th}. First, the nonnegative processes $u_\cdot$, $v_\cdot$ and $\gamma_\cdot$ appearing respectively in (H1), (H3) and (H4) need to satisfy
$$
\bigg\|\int_0^T u_t \dd t \bigg\|_\infty +\bigg\|\int_0^T v^2_t \dd t \bigg\|_\infty
+\bigg\|\int_0^T \left(\gamma_t+\gamma_t^{\frac{1}{1-\alpha}}+
\gamma_t^{\frac{2}{2-\alpha}}\right) \dif t \bigg\|_\infty
< +\infty.
$$
Second, the nonnegative processes $g^1_\cdot$ and $g^2_\cdot$ need to satisfies \eqref{eq:3.1*}. Third, the pair of solution processes $(y_t,z_t)_{t\in[0,T]}$ of BSDE (\ref{101}) need to satisfy that for each $\beta\in (0,1)$,
$$
\E\left[\sup\limits_{t\in[0,T]}|y_t|^\beta\right]+\E\left[\left(\int_0^T |z_t|^2{\rm d}t\right)^{\frac{\beta}{2}}\right]<+\infty,
$$
and that $y_\cdot$ belongs to class (D), and it further satisfies \eqref{eq:3.4} and \eqref{eq:3.5}.

\begin{rmk}\label{rm3.3} With respect to \cref{th}, we make the following remarks.
\begin{itemize}
\item [(i)] It is not difficult to verify that if the process $y_\cdot$ satisfies \eqref{eq:3.5}, then it must belong to $\s^\beta$ for each $\beta\in (0,1)$ and class (D), and it also has to satisfy \eqref{eq:3.4}. In light of Lemma 6.1 in \citet{Bri03}, the first assertion is evident. Now, let us prove the second one. In fact, if \eqref{eq:3.5} holds, then
    \begin{align*}
    \begin{split}
    \E\left[\int_0^T u_t|y_t| \dd t\right]&\leq \E\left[\int_0^T u_t|Y_t| \dd t\right]+\E\left[\int_0^T u_t \E[\eta|\F_t] \dd t\right]\\
    &\leq \E\left[\int_0^T u_t \dd t\cdot \sup_{t\in\T}|Y_t|\right]+\E\left[\int_0^T u_t \eta \dd t\right]\\
    &\leq \left\|\int_0^T u_t \dd t\right\|_\infty \left(\|Y\|_{\s^1}+\E[\eta]\right)<+\infty.
    \end{split}
    \end{align*}
  Consequently, Theorem \ref{th} can be reexpressed as follows: under the assumptions (H1)-(H5), BSDE \eqref{101} admits a unique solution $(y_t,z_t)_{t\in[0,T]}$ such that \eqref{eq:3.5} holds and $z_\cdot\in \M^\beta$ for each $\beta\in (0,1)$.

\item [(ii)] Assume that the process $u_\cdot$ is a deterministic function. We can verify that if $y_\cdot$ belongs to class (D), then \eqref{eq:3.4} holds automatically. In fact, by Fubini's Theorem we have
$$
\E\left[\int_0^T u_t|y_t|\dd t\right]=\int_0^T u_t\E\left[|y_t|\right] \dd t\leq\int_0^Tu_t\dd t\cdot \sup\limits_{t\in[0,T]}\E\left[|y_t|\right]<+\infty.
$$
This means that in the setting that $u_\cdot$ is independent of $\omega$, the condition \eqref{eq:3.4} in Theorem \ref{th} can be eliminated, which is exactly the cases in \citet{Fan18}, \citet{XiaoFanXu15} and \citet{Bri03}.

\item [(iii)] In light of (i) of \cref{rm3.1} and (ii) of \cref{rm3.3}, we draw a conclusion that \cref{th} extends and strengthens Theorems 6.2 and 6.3 of \cite{Bri03}, Theorem 4.1 of \cite{XiaoFanXu15} and Theorems 1 and 2 of \cite{Fan18}.

\item [(iv)] To the best of our knowledge, it seems to be the first time that the process $y_\cdot$ in the $L^1$ solution of BSDE \eqref{101} is explored to definitely satisfy \eqref{eq:3.5} and then \eqref{eq:3.4}.\vspace{0.2cm}
\end{itemize}
\end{rmk}

\begin{proof}[Outline of the proof of \cref{th}]
For the existence part, the method consists to assume first that the generator $g$ does not depend on the state variable $z$ and to show that a desired $L^1$ solution exists, i.e., there exists a pair of adapted processes $(y_t,z_t)_{t\in\T}$ with appropriate integrability properties such that
$$
y_t=\xi+\int_t^T g(s,y_s) {\rm d}s-\int_t^T z_s {\rm d}B_s, \ \ \ t\in[0,T].
$$
Then to show that the following scheme: $(y^0_t,z^0_t)_{t\in\T}:=(0,0)$ and for $n\geq 1$,
$$
y^n_t=\xi+\int_t^T g(s,y^n_s,z^{n-1}_s) {\rm d}s-\int_t^T z^n_s {\rm d}B_s, \ \ \ t\in[0,T]
$$
provides a sequence $(y^n_t,z^n_t)_{t\in\T}$ which is of Cauchy type in an appropriate space and converges to the desired $L^1$ solution of BSDE (\ref{101}). The idea mainly comes from \citet{Bri03} and \citet{Fan18}.
Some updated a priori estimates on adapted solutions of BSDEs and the maximum inequality on submartingales along with \cref{pro:2.3,pro:2.4} are combined together to address the stochastic coefficients and weaker integrability conditions. Particularly, verifying \eqref{eq:2.1} and \eqref{eq:2.2} and managing to apply the general stochastic Gronwall-type and Bihari-type inequalities are main difficulties which need us to overcome. For the uniqueness part, the method consists to apply the a priori estimate on $L^p$ solutions of BSDEs as well as the general stochastic Gronwall-type and Bihari-type inequalities. The idea mainly comes from Theorem 1 in \citet{Fan18}. The main difference lies in that the additional condition (\ref{eq:3.4}) is imposed in our proof in order to apply the stochastic Gronwall-type inequality-\cref{pro:2.3}. The detailed proof is deferred to \cref{sec:4-The proof of main result}.
\end{proof}

\subsection{Two examples\vspace{0.1cm}}

In this subsection, we provide two examples which \cref{th} can apply to, but any existing results can not apply to.
\begin{ex}\label{ex1}
Let $k=2$, $\ 0< T\leq +\infty$ and $\bar M>0$. Define the following stopping times:
$$\tau_{1}(\omega):=\inf \left\{t \in[0,T]: \int_{0}^{t}|B_{s}(\omega)| {\rm d}s \geq \bar M \right\} \wedge T$$
and
$$
\tau_{2}(\omega):=\inf \left\{t \in[0,T]: \int_{0}^{t} |B_{s}(\omega)|^2 {\rm d}s \geq \bar M\right\} \wedge T\vspace{0.1cm}
$$
with the convention ${\rm inf} \ \emptyset=+\infty$ and define the following processes:\vspace{0.1cm}
$$
\bar {u}_t(\omega):=|B_t(\omega)|{\mathbbm{1}}_{t\leq \tau_{1}(\omega)} \ \ {\rm and} \ \ \bar {v}_t(\omega):=|B_t(\omega)|{\mathbbm{1}}_{t\leq \tau_{2}(\omega)}, \ (\omega,t)\in \Omega \times [0,T].\vspace{0.1cm}
$$
It is obvious that $\bar {u}_\cdot \in{L^{\infty}(\Omega;L^{1}([0,T];\R_{+}))}$ and $\bar {v}_\cdot \in{L^{\infty}(\Omega;L^{2}([0,T];\R_{+}))}$. Let\vspace{0.1cm}
$$g(\omega,t,y,z):=\bar {u}_t(\omega)\begin{bmatrix}h(y_1)+e^{-y_2}\\ h(y_2)-e^{y_1}
\end{bmatrix}+\bar {v}_t(\omega)\begin{bmatrix}\sqrt[3]{1+|z|}\\ \ln (1+|z|)\end{bmatrix}, \ \ (\omega, t, y, z)\in\Omega\times[0,T]\times\R^2\times\R^{2\times d},$$
where
$$h(x):=x|{\rm ln}x|{\mathbbm{1}}_{0< x\leq \delta}+(h'(\delta-)(x-\delta)+h(\delta)){\mathbbm{1}}_{x>\delta}$$
with $\delta>0$ small enough.

It is not hard to verify that assumptions (H1)-(H4) and (i) of (H5) hold
with $u_\cdot:=\bar {u}_\cdot$, $\rho(\cdot):=h(\cdot)$, $v_\cdot=\gamma_\cdot:=\bar {v}_\cdot$, $g_\cdot^1:\equiv 0$, $g_\cdot^2:\equiv 1$ and $\alpha:=1/3$ for any $\xi\in L^1(\F_T;\R^k)$. In particular, we have
$$
\int_{0^+} \frac{\dd u}{h(u)}=-\int_{0^+} \frac{\dd u}{u\ln u}=-\lim\limits_{\epsilon\rightarrow 0} \ln|\ln u| \big|_0^\epsilon=+\infty.
$$
It follows from Theorem \ref{th} that BSDE (\ref{101}) admit a unique $L^1$ solution $(y_t,z_t)_{t\in[0,T]}$ satisfying (\ref{eq:3.4}).\vspace{0.2cm}
\end{ex}

\begin{ex}\label{ex2}
Let $\ 0< T\leq +\infty$ and define the processes $\bar {u}_\cdot, \bar {v}_\cdot$ as in \cref{ex1}. Let $y:=(y_1, y_2,..., y_k)$ and $g(t,y,z):=(g_1(t,y,z), g_2(t,y,z),\cdots,g_k(t,y,z))$, where for each $i=1,2,\cdots,k$ and each $(\omega, t, y, z)\in\Omega\times[0,T]\times\R^k\times\R^{k\times d}$,\vspace{-0.2cm}
$$
g_i(\omega,t,y,z):=\bar {u}_t(\omega)(\bar{h}(|y|)+e^{-|B_t(\omega)|y_i})+ \bar{v}_t(\omega)(1+|z|)^{2\over 3}+e^{-t}
$$
and
$$
\bar h(x):=-x|{\rm ln} x|^\beta {\mathbbm{1}}_{0<x\leq \delta}+(\bar h'(\delta-)(x-\delta)+\bar h(\delta)){\mathbbm{1}}_{x>\delta}
$$
with $\delta>0$ small enough and $\beta\in (0,1)$.

We can check that assumptions (H1)-(H5) hold with $u_\cdot:=\bar {u}_\cdot$, $\rho(\cdot):=\bar h(\cdot)$, $v_\cdot=\gamma_\cdot:=\bar {v}_\cdot$, $g_\cdot^1:\equiv 0$, $g_\cdot^2:\equiv 1$, $\alpha:=2/3\in [1/2,1)$ and $\bar p:=1/\beta>1$ for any $\xi\in L^1(\F_T;\R^k)$. In particular, we have
$$
\int_{0^+} \frac{u^{\bar p-1}\dd u}{\bar h^{\bar p}(u)}=-\int_{0^+} \frac{\dd u}{u\ln u}=-\lim\limits_{\epsilon\rightarrow 0} \ln|\ln u| \big|_0^\epsilon=+\infty.
$$
It follows from Theorem \ref{th} that BSDE (\ref{101}) admits a unique $L^1$ solution $(y_t,z_t)_{t\in[0,T]}$ satisfying (\ref{eq:3.4}).
\end{ex}

It should be especially emphasized that the conclusions in Examples \ref{ex1} and \ref{ex2} can not be obtained via any existing results including those established in \citet{Bri03}, \citet{XiaoFanXu15} and \citet{Fan18}, due to the elementary fact that $\bar u_\cdot$ and $\bar v_\cdot$ are two unbounded processes depending on $\omega$.

\section{Proof of the main result}
\label{sec:4-The proof of main result}
\setcounter{equation}{0}

The present section contains four subsections. The first subsection aims to establish several key propositions which will play important roles in the proof of Theorem \ref{th}, the next two subsections are devoted to the existence of Theorem \ref{th} and the last subsection deals with the uniqueness.

\subsection{Several key propositions\vspace{0.2cm}}

Let the constant $p>1$ and the process $u_\cdot$ be defined as in (H1). The following two assumptions on the generator of BSDEs are closely related to assumption (H1) and will be used later. They are respectively slight extensions of assumptions (H1a)$_p$ and (H1b)$_p$ in \citet{Fan15}. Moreover, the (H1a)$_p$ is just (H2)$_p$ of \citet{LiFan2023}, and (H2) of \citet{LiXuFan2021} for the case of $p=2$.

\begin{enumerate}
\renewcommand{\theenumi}{(H1a)$_p$}
\renewcommand{\labelenumi}{\theenumi}
\item  $g$ satisfies a $p$-order weak stochastic-monotonicity condition in $y$, i.e., there exists a process $u_.\in L^\infty(\Omega;L^1([0,T];\R_+))$ and a function $\kappa(\cdot)\in {\bf S}$ such that $\as, \forall y_1,y_2 \in \R^k,z\in\R^{k\times d},$
$$
|y_1-y_2|^{p-1}\left\langle {y_1-y_2\over |y_1-y_2|}\mathbbm{1}_{|y_1-y_2|\neq 0},\ \ g(\omega,t,y_1,z)-g(\omega,t,y_2,z)\right\rangle\leq u_t(\omega)\kappa(|y_1-y_2|^p).
$$
\end{enumerate}

\begin{enumerate}
\renewcommand{\theenumi}{(H1b)$_p$}
\renewcommand{\labelenumi}{\theenumi}
\item  $g$ satisfies a $p$-order stochastic Mao's condition in $y$, i.e., there exists a process $u_.\in L^\infty(\Omega;$ $L^1([0,T];\R_+))$ and a function $\varrho(\cdot)\in {\bf S}$ such that $\as, \forall y_1,y_2 \in \R^k,z\in\R^{k\times d},$
$$
\left\langle {y_1-y_2\over |y_1-y_2|}\mathbbm{1}_{|y_1-y_2|\neq 0},
    \ \ g(\omega,t,y_1,z)-g(\omega,t,y_2,z)\right\rangle\leq u_t(\omega)\varrho^{{1\over p}}(|y_1-y_2|^p).\vspace{0.2cm}
$$
\end{enumerate}

By an identical way with Proposition 1 of \citet{Fan15}, we can prove the following proposition which will be frequently used in the proof of \cref{th}. The only difference lies in that the process $u_\cdot$ in (H1a)$_p$ and (H1b)$_p$ is the constant $1$ in Proposition 1 of \citet{Fan15}. Its proof is omitted here.

\begin{pro}\label{pro:4.1}
It holds that
$$
{\rm (H1)} \Longleftrightarrow {\rm (H1a)}_1 \Longleftrightarrow {\rm (H1b)}_1 ,\ \ \ {\rm (H1b)}_p \Longrightarrow {\rm (H1)}\Longrightarrow {\rm (H1a)}_p\ {\rm for\ each}\ p>1,
$$
and
$$
{\rm (H1b)}_{\bar p} \Longrightarrow {\rm (H1b)}_p\ {\rm for\ each}\ \bar p>p>1.\vspace{0.2cm}
$$
In addition, if the generator $g$ satisfies assumption (H1) with a function $\rho(\cdot)$ satisfying \eqref{eq:3.2} in (ii) of (H4) for some $\bar p>1$, then it must satisfy (H1b)$_{\bar p}$ with some function $\varrho(\cdot)\in {\bf S}$.
\end{pro}

Since the functions $\rho(\cdot)$, $\kappa(\cdot)$ and $\varrho(\cdot)$ defined, respectively, in assumptions (H1), (H1a)$_p$ and (H1b)$_p$ are nondecreasing and concave with value $0$ at $0$, we can suppose that there exists a constant $A>0$ such that for each $x\in\R_+$,
\begin{equation}\label{eq:4.1}
\rho(x)\leq A(x+1),\ \ \ \kappa(x)\leq A(x+1)\ \ \ {\rm and}\ \ \ \varrho(x)\leq A(x+1).
\end{equation}
In addition, for convenience, we always suppose that $(u_t)_{t\in[0,T]}$ and $(v_t)_{t\in[0,T]}$ appearing, respectively, in assumptions (H1) and (H3), satisfy that
\begin{equation}\label{eq:4.2}
	\bigg\|\int_0^T (u_t+v^2_t) \dd t \bigg\|_\infty \leq M\vspace{0.1cm}
\end{equation}
for some constant $M>0$.\vspace{0.3cm}

In the sequel, let us further present several a priori estimates on the $L^p~(p>1)$ solution of BSDEs with unbounded stochastic coefficients, which will play crucial roles in the proof of \cref{th}. In stating them, the following assumptions (A1), (A2) and (A3) on the generator $g$ are required. We would like to mention that assumptions (H1a)$_2$ and (H3), (H1a)$_p$ and (H3), (H1b)$_p$ and (H3) are closely related to assumptions (A1), (A2) and (A3), respectively.

\begin{enumerate}
\renewcommand{\theenumi}{(A1)}
\renewcommand{\labelenumi}{\theenumi}
\item Assume $p>0$ and that $\as$, for each $(y,z)\in\R^k\times\R^{k\times d}$, we have
$$
 \left\langle y,\ g(\omega,t,y,z)\right\rangle \leq  \mu_t(\omega)|y|^2+\lambda_t(\omega)|y||z|+|y|f_t(\omega)+\varphi_t(\omega),
$$
where the processes $\mu_\cdot\in L^{\infty}\left(\Omega; L^{1}\left([0, T]; \mathbb{R}_{+}\right)\right)$, $\lambda_\cdot\in L^{\infty}\left(\Omega;L^{2}
\left([0, T]; \mathbb{R}_{+}\right)\right)$, $f_\cdot\in\mathcal{H}^p(0,T;\R)$ and $\varphi_\cdot\in\mathcal{H}^{p\over 2}(0,T;\R)$. Assume further that
$$
 \bigg\| \int_0^T (\mu_s+\lambda^2_s)\dd s  \bigg\|_\infty \leq M,\vspace{0.1cm}
$$
where the constant $M>0$ is defined in \eqref{eq:4.2}.
\end{enumerate}

\begin{enumerate}
\renewcommand{\theenumi}{(A2)}
\renewcommand{\labelenumi}{\theenumi}
\item Assume $p>1$ and that $\as$, for each $(y,z)\in\R^k\times\R^{k\times d}$, we have
$$
|y|^{p-1}\left\langle \frac{y}{|y|}\mathbbm{1}_{|y|\neq 0},\ \ g(\omega,t,y,z)\right\rangle \leq \mu_t(\omega)\psi(|y|^p)+\lambda_t(\omega)|y|^{p-1}|z|+|y|^{p-1}f_t(\omega),
$$
where the processes $\mu_\cdot$, $\lambda_\cdot$ and $f_\cdot$ are defined in (A1), and $\psi(\cdot): \mathbb{R}_{+}\rightarrow\mathbb{R}_{+}$ is a nondecreasing concave function with $\psi(0)=0.$
\end{enumerate}

\begin{enumerate}
\renewcommand{\theenumi}{(A3)}
\renewcommand{\labelenumi}{\theenumi}
\item  Assume $p>1$ and that $\as$, for each $(y,z)\in\R^k\times\R^{k\times d}$, we have
$$
 \left\langle {y\over |y|}\mathbbm{1}_{|y|\neq 0}, \ \ g(\omega,t,y,z)\right\rangle\leq
    \mu_t(\omega)\phi^{{1\over p}}(|y|^p)+\lambda_t(\omega)|z|+f_t(\omega),
$$
where the processes $\mu_\cdot$, $\lambda_\cdot$ and $f_\cdot$ are defined in (A1), and $\phi(\cdot): \mathbb{R}_{+} \rightarrow \mathbb{R}_{+}$ is a nondecreasing concave function with $\phi(0)=0$.
\end{enumerate}

The first a priori estimate is the following \cref{pro:4.2}, which comes from Lemma 2.5 in \citet{LiFan2023}. This proposition is usually used to verify $z_\cdot\in\M^p$ when $y_\cdot\in \s^p$ for some $p>0$, and to verify the convergence of the sequence of $\{z^n_\cdot\}$ in the space of $\M^p$ when the convergence of the sequence of $\{y^n_\cdot\}$ in the space of $\s^p$ is in force for some $p>0$.

\begin{pro}\label{pro:4.2}
Let $p>0$, assumption (A1) hold and $(y_t,z_t)_{t\in[0,T]}$ be a solution of BSDE \eqref{101} such that $y_\cdot\in \s^p(0,T;\R^k)$. Then, $z_\cdot\in\M^p(0,T;\R^{k\times d})$ and there exists a positive constant $C_p$ depending only on $p$ such that for each $0\leq u \leq t\leq T$,
$$
\begin{aligned}
\mathbb{E}\left[\left(\int_{t}^{T}\left|z_{s}\right|^{2} \mathrm{ d} s\right)^{\frac{p}{2}} \bigg|{\F}_{u}\right] & \leq C_{p}(1+M)^{\frac{p}{2}} \mathbb{E}\left[\sup _{s \in[t, T]}\left|y_{s}\right|^{p} \bigg|{\F}_{u}\right] \\
&\ \ \ +C_{p} \mathbb{E}\left[\left(\int_{t}^{T} f_{s} \mathrm{ d} s\right)^{p} \bigg|{\F}_{u}\right]+C_{p} \mathbb{E}\left[\left(\int_{t}^{T} \varphi_{s} \mathrm{ d} s\right)^{\frac{p}{2}} \bigg|{\F}_{u}\right].
\end{aligned}
$$
\end{pro}

The second a priori estimate is the following \cref{pro:4.5}, which comes from Lemma 2.6 in \citet{LiFan2023}. This assertion is usually used to prove the uniqueness of the process $y_\cdot$ and the convergence of the sequence of $\{y^n_\cdot\}$ in the space of $\s^p$ for some $p>1$ under the $p$-order weak stochastic-monotonicity condition (H1a)$_p$, by combining with the stochastic Bihari-type inequality.

\begin{pro}\label{pro:4.5}
Let $p>1$, assumption (A2) hold and $(y_t,z_t)_{t\in[0,T]}$ be an $L^p$ solution of BSDE \eqref{101}. Then there exists a constant $C_{p,M}>0$ depending only on $p$ and $M$ such that for each $0\leq u \leq t\leq T$,
$$
\mathbb{E}\left[\sup _{s \in[t, T]}\left|y_{s}\right|^{p} \bigg|{\F}_{u}\right] \leq C_{p, M} \mathbb{E}\left[|\xi|^{p}+\int_{t}^{T} \mu_{s} \psi\left(\left|y_{s}\right|^{p}\right) \mathrm{d} s+\left(\int_{t}^{T} f_{s} \mathrm{ d} s\right)^{p} \bigg|{\F}_{u}\right].
$$
\end{pro}

The following \cref{pro:4.6} comes from Proposition 5.1 in \citet{LiXuFan2021}, which is the third a priori estimate. It is used to verify, by virtue of the stochastic Bihari-type inequality, the convergence of the sequence of $\{(y^n_\cdot,z^n_\cdot)\}$ in the space of $\s^2\times\M^2$ under the stochastic one-sided Osgood condition (H1) or the weak stochastic-monotonicity condition (H1a)$_2$.

\begin{pro}\label{pro:4.6}
Let (A2) hold with $p=2$ and $(y_t,z_t)_{t\in[0,T]}$ be an $L^2$ solution of BSDE\eqref{101}. Then there exists a constant $C_M>0$ depending only on $M$ such that for each $0\leq u \leq t\leq T$,
$$
\mathbb{E}\left[\sup\limits_{s\in[t,T]}\left|y_{s}\right|^{2} \bigg|{\F}_{u}\right] +\E\left[\left.\Dis \int_t^T
|z_s|^2{\rm d}s\right|\F_u\right] \leq C_M \mathbb{E}\left[|\xi|^{2}+\int_{t}^{T}\mu_{s}
\psi(|y_{s}|^{2})\mathrm{d}s+\left(\int_{t}^{T} f_{s}\mathrm{ d} s\right)^{2} \bigg|{\F}_{u}\right].
$$
\end{pro}

Finally, enlightened by Propositions 1-2 of \citet{FanJiang19}, we can prove the following \cref{pro:4.7}, which is the fourth a priori estimate. It is usually used to prove, by virtue of the stochastic Bihari-type inequality, the convergence of the sequence of $\{(y^n_\cdot,z^n_\cdot)\}$ in the space of $\s^p\times\M^p$ for some $p\in (1,2)$ under the $p$-order stochastic Mao's condition (H1b)$_p$, especially for handling the case of $\alpha\in [1/2,1)$ (see (ii) of (H4) in \cref{sec:3-The statement of Main result}) in \cref{th}.

\begin{pro}\label{pro:4.7}
Let $p>1$, assumption (A3) hold and $(y_t,z_t)_{t\in[0,T]}$ be a solution of BSDE \eqref{101} such that $y_\cdot\in \s^p(0,T;\R^k)$. Then $z_\cdot\in\M^p(0,T;\R^{k\times d})$ and there exists a positive constant $C_{p,M}$ depending only on $(p,M)$ such that for each $0\leq u \leq t\leq T$,
$$
\mathbb{E}\left[\sup\limits_{s\in[t,T]}\left|y_{s}\right|^{p}\bigg|{\F}_{u}\right] +\E\left[\left(\Dis \int_t^T
|z_s|^2{\rm d}s\right)^\frac{p}{2}\bigg|\F_u\right]
\leq C_{p,M} \mathbb{E}\left[|\xi|^{p}+\int_{t}^{T}\mu_{s}\phi(|y_{s}|^{p})\mathrm{d}s + \left(\int_{t}^{T} f_{s}\mathrm{d}s\right)^{p}\bigg|{\F}_{u}\right].
$$
\end{pro}

To keep the main proof of \cref{th} uncluttered, we postpone this somewhat technical proof of \cref{pro:4.7} to the Appendix. In addition, we would like to emphasize that if the generator $g$ only satisfies assumption (A2) for some $p>1$ and $(y_t,z_t)_{t\in[0,T]}$ is an $L^p$ solution of BSDE \eqref{101}, then the a priori estimate of the $L^p$ solution obtained in \cref{pro:4.7} does not hold generally.

\subsection{Proof of the existence part of \cref{th}: the case of $g$ being independent of $z$\vspace{0.2cm}}

In this subsection, we will prove the following \cref{pro:4.3} which answers the existence part of \cref{th} in the case of the generator $g$ being independent of the state variable $z$. The idea of the proof mainly comes from Proposition 5 of \citet{Fan18}. Furthermore, inspired by \citet{LiXuFan2021} and \citet{LiFan2023}, we need to apply \cref{pro:2.3,pro:2.4} in order to address the stochastic one-sided Osgood condition of the generator $g$ in the state variable $y$. However, when utilizing \cref{pro:2.4}, verification of the condition \eqref{eq:2.2} is a challenge which needs us to overcome in the subsequent proof. In addition, in our proof the way utilizing \cref{pro:2.4} and the maximum inequality on submartingales to verify that the constructed sequence of processes converges uniformly in probability and is a Cauchy sequence under the norm $\|\cdot\|_{1}$ also seems to be innovative.

\begin{pro}\label{pro:4.3}
Assume that the generator $g$ is independent of the state variable $z$ and assumptions (H1), (H2) and (H5) hold. Then, BSDE \eqref{101} admits an $L^1$ solution $(y_t,z_t)_{t\in[0,T]}$ satisfying (\ref{eq:3.4}).\vspace{-0.2cm}
\end{pro}

\begin{proof}
For each $ x\in \R^k, r\in\R_+$ and $n\geq1$, define the function $q_r(x):=\frac{rx}{r \vee |x|}$, and let
$$\xi^n:=q_n(\xi) \ \ {\rm and} \ \ g^n(t,y):=g(t,y)-g(t,0)+q_{ne^{-t}}(g(t,0)).$$
Then, for each $n\geq1$ and $t\in[0,T]$, we have
$$ |\xi^{n}|\leq n \ \ {\rm and} \ \ |g^n(t,0)|\leq ne^{-t}.$$
It is straightforward to verify that for each $n\geq1$, $g^n$ satisfies assumptions (H1), (H2) and
$$\E\left[|\xi^n|^2+\left(\int_0^T|g^n(t,0)| \dd t\right)^2\right]<+\infty.$$
By \cref{pro:4.1} we know that $g^n$ also satisfies assumptions (H1a)$_2$ for each $n\geq1$. It then follows from Theorem 3.2 in \citet{LiXuFan2021} that the BSDE
\begin{equation}\label{3}
y^n_t=\xi^n+\int^{T}_{t}g^n(s,y^n_s)\,\dd s-\int^{T}_{t}z_s^n\,\dd B_s, \ \ t\in[0,T]
\end{equation}
admits a unique $L^2$ solution $(y_t^n,z_t^n)_{t\in[0,T]}$. For each $n,i\geq1$, we set $$\hat{y}^{n,i}_\cdot:=y^{n+i}_\cdot-y^n_\cdot, \ \ \  \hat z^{n,i}_\cdot:=z^{n+i}_\cdot-z^n_\cdot \ \ \ \text{and} \ \ \ \hat\xi^{n,i}:=\xi^{n+i}-\xi^{n}.$$

The following proof will be divided into four steps.\vspace{0.2cm}

\textbf{First step.} We prove that the sequence of random variables $\{\supi\supt|\hat y_t^{n,i}|^\beta\}_{n=1}^{+\infty}$ is a uniformly integrable for each $\beta\in(0,1)$.\vspace{0.2cm}

First, in light of (\ref{3}), by Corollary 2.3 in \citet{Bri03}, the definition of $g^n(t,y)$ and the assumption (H1) we can deduce that for each $n,i\geq1$,
\begin{equation}\label{4}
|\hat y_t^{n,i}|\leq H_n(t)+\E\left[\int_t^T u_s \rho(|\hat y^{n,i}_s|)\,\dd s \bigg| \F_t \right], \ \ t\in[0,T],
\end{equation}
where
$$
H_n(t):=\E\left[|\xi|\mathbbm{1}_{|\xi|>n} + \int_0^T |g(s,0)|\mathbbm{1}_{|g(s,0)|>ne^{-s}}\dd s \bigg|\F_t \right].\vspace{0.1cm}
$$
It then follows from \eqref{eq:4.1} and \eqref{4} that
\begin{equation}
|\hat y_t^{n,i}|\leq \E\left[\zeta_n+\int_t^T u_s |\hat y^{n,i}_s| \dd s \bigg| \F_t \right], \ \ t\in[0,T],\nonumber
\end{equation}
where
$$
\zeta_n:=|\xi|\mathbbm{1}_{|\xi|>n} + \int_0^T |g(s,0)|\mathbbm{1}_{|g(s,0)|>ne^{-s}}\dd s + A\int_0^T u_s \dd s.\vspace{0.2cm}
$$
Since $y_\cdot^n \in \s^2$ for each $n\geq1$ and $u.\in L^\infty(\Omega;L^1([0,T];\R_+))$, by virtue of the stochastic Gronwall-type inequality (see \cref{pro:2.3}) we can get that for each $n,i\geq 1$ and $t\in\T$,
\begin{equation}\label{5}
|\hat y_t^{n,i}| \leq \Dis e^{A\|\int_0^Tu_s \dd s\|_\infty}\left(A\E\left[\displaystyle\int_0^T u_s \dd s \bigg| \F_t\right]+H_n(t)\right)\leq \Dis e^{AM}\left(AM+H_n(t)\right).
\end{equation}

Next, let us consider the integrability of $H_n(\cdot)$. In fact, Lemma 6.1 in \citet{Bri03} along with assumption (H5) yields that for each $\beta\in(0,1)$,
\begin{equation}\label{6}
\begin{array}{lll}
\supn\E[\supt|H_n(t)|^\beta] & \leq & \Dis \frac{1}{1-\beta} \supn(\E[H_n(T)])^\beta\vspace{0.1cm}\\
& \leq & \Dis \frac{1}{1-\beta} \left(\E\left[|\xi| + \displaystyle\int_0^T |g(s,0)|\dd s \right]\right)^\beta < +\infty.
\end{array}
\end{equation}
Combining \eqref{5} and \eqref{6}, we obtain that for each $\beta\in(0,1)$,
\begin{equation}\label{7}
\begin{array}{lll}
\supn\E\left[\supi\supt|\hat{y}_t^{n,i}|^\beta\right] \leq \supn\E\left[e^{\beta AM}\left(AM+\supt H_n(t)\right)^\beta\right] < +\infty,
\end{array}
\end{equation}
which indicates that the sequence of random variables $\{\supi\supt |\hat{y}_t^{n,i}|^{\beta'}\}_{n=1}^{+\infty}$ is uniformly integrable for each $\beta'\in (0,\beta)$. Thus, the desired assertion follows \vspace{0.3cm}immediately.

\textbf{Second step.} \ We prove that $\{y^n_\cdot\}_{n=1}^{+\infty}$ is a Cauchy sequence in $\s^\beta$ for each $\beta\in(0,1)$ and also converges under the norm $\|\cdot\|_1$.\vspace{0.2cm}

In light of the definition of $H_n(\cdot)$ and the assumption (H5), by taking first supremum with respect to $i$ and then limit superior with respect to $n$ on both sides of (\ref{4}) we have
\begin{equation}\label{8}
\begin{array}{lll}
\varlimsup\limits_{ n\rightarrow\infty}\supi|\hat y_t^{n,i}| & \leq &\varlimsup\limits_{n\rightarrow\infty}H_n(t) + \varlimsup\limits_{n\rightarrow\infty}\supi\E\left[\displaystyle\int_t^T u_s\rho(|\hat y_s^{n,i}|) \dd s\bigg|\F_t\right]\vspace{0.1cm}\\
& = & \varlimsup\limits_{n\rightarrow\infty}\supi\E\left[\displaystyle\int_t^T u_s\rho(|\hat y_s^{n,i}|) \dd s\bigg|\F_t\right], \ \ t\in[0,T].
\end{array}
\end{equation}
A combination of \eqref{eq:4.1}, (\ref{5}), the definition of $H_n(\cdot)$, Fubini's theorem and assumption (H5) leads to
\begin{equation}\label{11}
\begin{array}{lll}
\E\left[\Dis \int_0^T u_s \rho(\supn\supi|\hat y^{n,i}_s|)\,\dd s \right]
& \leq & \Dis A\E\left[\int_0^T u_s\, \supn\supi|\hat y^{n,i}_s| \dd s \right]+A\E\left[\int_0^T u_s \dd s\right] \vspace{0.2cm}\\
& \leq & \Dis Ae^{AM}\E\left[\int_0^T u_s \left(AM+\supn H_n(s)\right)\dd s \right]+AM\vspace{0.2cm}\\
 & = & \Dis Ae^{AM}\E\left[\int_0^T u_s \, \supn H_n(s)\dd s \right]+\bar C_{A,M}\vspace{0.2cm}\\
 & \leq & \Dis
 Ae^{AM}\E\left[\int_0^T u_s \E\left[|\xi|+\int_0^T|g(t,0,0)| \dd t\bigg|\F_s\right] \dd s \right]+\bar C_{A,M}\vspace{0.2cm}\\
  & = &\Dis
 Ae^{AM}\E\left[\int_0^T u_s\left(|\xi|+\int_0^T|g(t,0,0)| \dd t\right) \dd s\right]+\bar C_{A,M}\vspace{0.2cm}\\
  & \leq &\Dis
 AMe^{AM} \E\left[|\xi|+\int_0^T|g(t,0,0)| \dd t\right]+\bar C_{A,M}<+\infty,
\end{array}
\end{equation}
where
$$
\bar C_{A,M}:=A^2M^2e^{AM}+AM.
$$
Now, we set
$$h_t:=\varlimsup\limits_{n\rightarrow\infty}\supi|\hat y_t^{n,i}|.$$
Combining (\ref{8}), (\ref{11}), Fatou's lemma and the continuity and monotonicity of $\rho(\cdot)$ yields that
$$
h_t\leq\E\left[\int_t^T u_s \rho(h_s)\dd s \bigg| \F_t\right],\ \ \ t\in[0,T].
$$
and
$$
\E\left[\int_0^T u_s h_s\dd s\right]\leq \E\left[\int_0^T u_s\, \supn\supi|\hat y^{n,i}_s| \dd s \right]<+\infty.\vspace{0.1cm}
$$
Then, by the stochastic Bihari-type inequality (see \cref{pro:2.4}) we can conclude that
\begin{equation}\label{13}
h_t=\varlimsup\limits_{n\rightarrow\infty}\supi|\hat y_t^{n,i}|\equiv 0, \ \ \ t\in[0,T],
\end{equation}
which indicates that for each fixed $t\in[0,T]$, \ $\supi|\hat y_t^{n,i}|\rightarrow 0$ as $n\rightarrow\infty$.\vspace{0.2cm}

Furthermore, in light of (\ref{11}), by taking first the supremum in $i$ and then $t$ on both sides of (\ref{4}) and using the maximum inequality on submartingales we can deduce that for each $\varepsilon>0$ and $n\geq 1$,
\begin{equation}\label{14}
\begin{array}{l}
\mathbbm{P}\left(\left\{\supt\supi|\hat y^{n,i}_t| \geq \varepsilon\right\}\right)\vspace{0.1cm}\\
\ \ \leq \mathbbm{P}\left(\left\{\supt\left(H_n(t)+\E\left[\displaystyle\int_0^T u_s \rho(\supi |\hat y^{n,i}_s|)\,\dd s \bigg| \F_t \right]\right)\geq \varepsilon\right\}\right)\vspace{0.1cm}\\
\ \ \leq \displaystyle\frac{1}{\varepsilon} \E\left[H_n(T)+\displaystyle\int_0^T u_s \rho(\supi |\hat y^{n,i}_s|)\,\dd s\right].
\end{array}
\end{equation}
Now, in light of (H5), the monotonicity and continuity of $\rho(\cdot)$ with $\rho(0)=0$,  (\ref{13}) and (\ref{11}), letting $n\rightarrow\infty$ in (\ref{14}) and using Lebesgue's dominated convergence theorem yields that for each $\varepsilon>0$,
\begin{equation}\label{15}
\limn\mathbbm{P}\left(\left\{\supt\supi|\hat y^{n,i}_s| \geq \varepsilon\right\}\right)= 0.
\end{equation}
Thus, by virtue of (\ref{7}) and (\ref{15}) we can conclude that
\begin{equation}\label{16}
\limn\supi\E\left[\supt|\hat{y}_{t}^{n,i}|^{\beta}\right]=\limn\E\left[\supt\supi|\hat{y}_{t}^{n,i}|^{\beta}\right] =0, \ \ \ \forall \beta\in(0,1).\vspace{0.1cm}
\end{equation}
That is, $\{y_\cdot^n\}_{n=1}^{+\infty}$ is a Cauchy sequence in $\s^{\beta}$ for each $\beta\in (0,1)$.\vspace{0.2cm}

Finally, we verify the convergence under the norm $\|\cdot\|_1$. It follows from (\ref{4}) that
\begin{equation}\label{31}
\E\left[|\hat y_t^{n,i}|\right]\leq c_n+\E\left[\int_0^T u_s \rho(|\hat y^{n,i}_s|)\,\dd s \right],
\end{equation}
where
$$
c_n :=\E\left[|\xi|\mathbbm{1}_{|\xi|>n} + \int_0^T |g(s,0)|\mathbbm{1}_{|g(s,0)|>ne^{-s}}\dd s \right].
$$
Then, in light of (H5), the monotonicity and continuity of $\rho(\cdot)$ with $\rho(0)=0$, (\ref{13}), (\ref{11}) and Lebesgue's dominated convergence theorem, by taking first supremum with respect to $t$ and $i$, then limit with respect to $n$ on both sides of (\ref{31}), we obtain that
\begin{equation}\label{17}
\begin{array}{lll}
0\leq \lim\limits_{n\rightarrow\infty}\supi\supt\E[|\hat{y}_{t}^{n,i}|] & \leq & \lim\limits_{n\rightarrow\infty}c_n+\lim\limits_{n\rightarrow\infty}\supi\E\left[\displaystyle\int_0^T u_s \rho(|\hat y^{n,i}_s|)\,\dd s \right]\vspace{0.1cm}\\
& \leq & \lim\limits_{n\rightarrow\infty}\E\left[\displaystyle\int_0^T u_s \rho(\supi|\hat y^{n,i}_s|)\,\dd s \right]= 0,
\end{array}
\end{equation}
which means that $\{y_\cdot^n\}_{n=1}^{+\infty}$ is also a Cauchy sequence under the norm $\|\cdot\|_{1}$.\vspace{0.2cm}

{\bf Third step.} \ We prove that $\{z_\cdot^n\}_{n=1}^{+\infty}$ is a Cauchy sequence in $\M^\beta$ for each $\beta\in(0,1)$.\vspace{0.2cm}

Note that for each $n,i\geq 1$, $(\hat y_t^{n,i},\hat z_t^{n,i})_{t\in[0,T]}$ is an $L^2$ solution of the following BSDE:
\begin{equation}
\hat{y}_t^{n,i}=\hat \xi^{n,i}+\int^{T}_{t}\hat g^{n,i}(s,\hat y_s^{n,i})\dd s-\int^{T}_{t}\hat{z}_s^{n,i}\dd B_s, \quad t\in[0,T],
\end{equation}
in which, $\as$, for each $n,i\geq 1$ and $y\in\R^k$,
$$\hat g^{n,i}(t,y):=g^{n+i}(t,y+y^n_t)-g^n(t,y^n_t).$$
By virtue of the definitions of $\hat g^{n,i}$ and $g^n$, the assumption (H1) on the generator $g$ and \cref{pro:4.1}, it is not hard to verify that for each $n,i\geq 1$, the generator $\hat g^{n,i}$ satisfies the assumption (H1a)$_2$ with a linearly growing function $\kappa(\cdot)\in {\bf S}$ satisfying that for each $m\geq1$,
\begin{equation}\label{115}
\kappa(x)\leq (m+A)x+\kappa\left(\frac{A}{m}\right), \ \ x\in\R^+.
\end{equation}
In fact, if $0\leq x\leq\frac{A}{m}$, (\ref{115}) is obvious, otherwise if $ x>\frac{A}{m}$, then $\kappa(x)\leq Ax+A\leq Ax+mx$.
Then, we deduce that for each $m,n,i\geq 1$ and $\beta\in (0,1)$, the generator $\hat{g}^{n,i}$ satisfies the assumption (A1) with
$$p:=\beta,\ \  \mu_\cdot:=(m+A)u_\cdot,\ \  \lambda_\cdot:\equiv 0, \ \ f_t:=|g(t,0)|\mathbbm{1}_{|g(t,0)|>ne^{-t}},\ \  \varphi_\cdot:=\kappa\left(\frac{A}{m}\right)u_\cdot.$$
Since $\hat{y}_\cdot^{n,i}\in \s^\beta$ for each $\beta\in (0,1)$, by \cref{pro:4.2} with $u=t=0$ we know that there exists a constant $C_\beta>0$ depending only on $\beta$ such that for each $m,n,i\geq 1$,
\begin{equation}\label{21}
\begin{array}{l}
\E\left[\left(\Dis \int_0^T |\hat z_t^{n,i}|^2 \dd t\right)^{\frac{\beta}{2}}\right]\\
\ \ \leq C_\beta(1+(m+A)M)^{\frac{\beta}{2}}\E\left[\supt|\hat y_t^{n,i}|^\beta\right]+
C_\beta\Dis \left(\kappa\left(\frac{A}{m}\right)\right)^{\frac{\beta}{2}}\E\left[\left(\Dis \int_0^T u_t \dd t\right)^{\frac{\beta}{2}}\right]\vspace{0.1cm}\\
\ \ \ \ + C_\beta\E\left[\left(\Dis \int^{T}_{0}|g(t,0)|\mathbbm{1}_{|g(t,0)|>ne^{-t}}\dd t\right)^\beta\right].
\end{array}
\end{equation}
Then, in light of (\ref{16}), the assumption (H5) on the generator $g$ and Lebesgue's dominated convergence theorem along with the continuity of $\kappa(\cdot)$ with $\kappa(0)=0$, by taking first supremum on $i$, then letting $n\rightarrow\infty$ and finally $m\rightarrow\infty$ in (\ref{21}), we can conclude that for each $\beta\in (0,1)$,
\begin{equation}\label{22}
\limn\supi\E\left[\left(\int_0^T |\hat z_t^{n,i}|^2 \dd t\right)^{\frac{\beta}{2}}\right]=0,
\end{equation}
which means that $\{z_\cdot^{n}\}_{i=1}^{+\infty}$ is a Cauchy sequence in space $\M^{\beta}$ for each $\beta\in (0,1)$.\vspace{0.1cm}

Finally, in light of (\ref{16}), (\ref{17}) and (\ref{22}) along with (H2), sending $n\rightarrow\infty$ in (\ref{3}) yields that BSDE \eqref{101} admits an $L^1$ solution $(y_t,z_t)_{t\in[0,T]}$, which is the limit  of $\{(y^n_\cdot,z^n_\cdot)\}_{n=1}^{+\infty}$ in $\bigcap_{\beta\in(0,1)}\s^\beta\times\M^\beta$.\vspace{0.2cm}

{\bf Fourth step.} We prove that the process $y_\cdot$ satisfies \eqref{eq:3.4} and \eqref{eq:3.5}.\vspace{0.2cm}

Note that $(y^n_t,z^n_t)_{t\in\T}$ is the unique $L^2$ solution of BSDE (\ref{3}). A computation similar to (\ref{4}) yields that for each $n\geq 1$ and $t\in [0,T]$,
$$
|y^n_t|\leq H(t)+\E\left[\left.\int_t^T u_s\rho(y^n_s) \dd s\right|\F_t\right],
$$
where
$$
H(t):=\E\left[\left. |\xi|+\int_0^T |g(t,0)| {\rm d}t  \right|\F_t\right].\vspace{0.1cm}
$$
And, by an analysis similar to (\ref{5}) in the first step, we can deduce that for each $n\geq 1$,
$$
|y^n_t|\leq e^{AM}\left(AM+ H(t)\right),\  \  t\in [0,T].
$$
It then follows that
$$
|y_t|\leq \sup\limits_{n\geq 1}|y^n_t|\leq e^{AM}\left(AM+ \E\left[\left. |\xi|+\int_0^T |g(t,0)| {\rm d}t  \right|\F_t\right]\right),\  \  t\in [0,T].
$$
Consequently, \eqref{eq:3.5} holds and then, in light of (i) of \cref{rm3.3}, we know that the process $y_\cdot$ also satisfies \eqref{eq:3.4}. The proof of \cref{pro:4.3} is then complete.
\end{proof}

\begin{rmk}
To the best of our knowledge, the assertion that the above-constructed process $y_\cdot$ satisfies \eqref{eq:3.4} and \eqref{eq:3.5} is explored at the first time. Based on this assertion, in the next subsection we will prove that the same conclusion holds for the general case that the generator $g$ may depend on the variable $z$.\vspace{0.2cm}
\end{rmk}


\subsection{Proof of the existence part of \cref{th}: the general case\vspace{0.2cm}}

Now, we are at a position to prove the general existence of the $L^1$ solution in \cref{th}. The idea of the proof mainly comes from Theorem 2 of \citet{Fan18}. In comparison with Theorem 2 of \citet{Fan18}, managing to verify the condition \eqref{eq:2.2} of \cref{pro:2.4} is one of the main difficulties in our proof when utilizing the stochastic Bihari-type inequality. Moreover, inspired by the proof of Theorem 3.1 in \citet{LiuLiFan2020}, a technique utilizing stopping times to subdivide the interval $\T$ is also applied to obtain the desired estimate on the underlying subintervals.

\begin{proof}[Proof of the existence part of \cref{th}]
According to Picard's iteration method, we set $(y_\cdot^0,z_\cdot^0):=(0,0)$ and define the sequence $\{(y_\cdot^n,z_\cdot^n)\}_{n=1}^{+\infty}$ by the $L^1$ solution of the following BSDE:
\begin{equation}\label{18}
y^n_t=\xi+\int^{T}_{t}g(s,y^n_s,z^{n-1}_s)\,\dd s-\int^{T}_{t}z_s^n\,\dd B_s, \ \ t\in[0,T].
\end{equation}
Since assumptions (H1)-(H5) hold for $g$ and $\xi$, it is not very hard to verify that for each integer $n\geq 1$ and process $z_\cdot^{n-1}\in \bigcap_{\beta\in (0,1)}\M^\beta$, the generator $g(t,y,z_t^{n-1})$ satisfies assumptions (H1), (H2) and (H5). Then, by \cref{pro:4.3} we can conclude that for each $n\geq 1$, BSDE (\ref{18}) admits an adapted solution $(y_t^n,z_t^n)_{t\in[0,T]}\in\bigcap_{\beta\in(0,1)}\s^\beta\times \M^\beta$ such that
\begin{equation}\label{eq:4.24}
(y_t^n)_{t\in[0,T]}\ {\rm belongs\ to\ class}\ (D)\ {\rm and}\ \E\left[\int_0^T u_t|y^n_t| \dd t\right]<+\infty.
\end{equation}

In the sequel, let us talk about the convergence of the sequence
$\{(y_\cdot^n,z_\cdot^n)\}_{n=1}^{+\infty}$. To begin with, for each $n,i\geq 1$, we set
$$\hat y_\cdot^{n,i}:=y_\cdot^{n+i}-y_\cdot^n\ \ \ {\rm and}\ \ \ \hat z_\cdot^{n,i}:
=z_\cdot^{n+i}-z_\cdot^n.\vspace{0.1cm}$$
Then, it is straightforward to verify that $(\hat y_t^{n,i},\hat z_t^{n,i})_{t\in[0,T]}$ solves the following BSDE:
\begin{equation}\label{33}
\hat y^{n,i}_t=\int^{T}_{t}\bar g^{n,i}(s,\hat y^{n,i}_s)\dd s-\int^{T}_{t}\hat z_s^{n,i}\dd B_s, \ \ t\in[0,T],
\end{equation}
where $\as$, for each $n,i\geq 1$ and $y\in\R^k$,
$$\bar g^{n,i}(t,y):=g(t,y+y_t^n,z_t^{n+i-1})-g(t,y_t^n,z_t^{n-1}).\vspace{0.1cm}$$

The following proof is divided into two steps.\vspace{0.2cm}

\textbf{First step.} We prove that for each $n,i\geq 1$,  $\hat{y}^{n,i}_\cdot\in\s^\frac{\beta}{\alpha}$ for each $\beta\in (\alpha,1)$.\vspace{0.2cm}

Let us fix $n,i\geq 1$ arbitrarily. For each $m\geq1$, denote the following stopping time:
$$
\sigma_m:=\inf\left\{t\in[0,T]:\int^t_{0}(|z_s^{n+i-1}|^2+|z_s^{n-1}|^2 )
\,\mathrm{d} s\geq m\right\}\wedge T.
$$
It follows from Corollary 2.3 in \citet{Bri03} and assumptions (H1) and (H4) that for each $m\geq 1$,
\begin{equation}\label{104}
|\hat{y}^{n,i}_{t\wedge\sigma_m}|\leq\E\left[|\hat{y}^{n,i}_{\sigma_m}|
+\left.\int^{\sigma_m}_{t\wedge\sigma_m} u_s\rho(|\hat{y}^{n,i}_s|)\ \mathrm{d}s\right|\F_t\right]+G^{n,i}(t), \ \ t\in[0,T],
\end{equation}
where
$$
G^{n,i}(t):=2\E\left[\left.\int^{T}_{0}\gamma_s\left(g^1_s+g^2_s+
|y^n_s|+|z^{n-1}_s|+|z^{n+i-1}_s|\right)^\alpha\ \mathrm{d}s\right|\F_t\right].\vspace{0.1cm}
$$
In light of \eqref{eq:4.1} and \eqref{eq:4.24}, by taking limit with respect to $m$ on  both sides of (\ref{104}) and using the Lebesgue's dominated convergence theorem we get that
\begin{equation}\label{105}
|\hat{y}^{n,i}_t|\leq G^{n,i}(t)+\E\left[\left.\int_t^T u_s \rho(|\hat{y}^{n,i}_s|)\ \mathrm{d}s\right|\F_t\right], \   \  t\in [0,T].
\end{equation}
Next, we consider the integrability of $G^{n,i}(\cdot)$. In light of $\beta\in (\alpha, 1)$, it follows from Doob's maximum inequality on submartingales that there exists a constant $C_\beta^\alpha>0$ depending only on $\beta$ and $\alpha$ such that
$$
\E\left[\supt|G^{n,i}(t)|^{\frac{\beta}{\alpha}}\right] \leq C_\beta^\alpha\E\left[\left(\int_0^T \gamma_s\left(g^1_s+g^2_s+
|y^n_s|+|z^{n-1}_s|+|z^{n+i-1}_s|\right)^\alpha\ \mathrm{d}s\right)^{\frac{\beta}{\alpha}}\right].
$$
By H\"{o}lder inequality we have
$$
\E\left[\left(\int_0^T \gamma_s (g^1_s)^\alpha \dd s\right)^{\frac{\beta}{\alpha}}\right] \leq \left\|\int_0^T \gamma_s^{\frac{1}{1-\alpha}}\dd s\right\|_\infty^{\frac{\beta(1-\alpha)}{\alpha}}\E\left[\left(\Dis \int_0^T g^1_s \dd s\right)^\beta\right],
$$
$$
\E\left[\left(\int_0^T \gamma_s (g^2_s)^\alpha \dd s\right)^{\frac{\beta}{\alpha}}\right] \leq \left\|\int_0^T \gamma_s\dd s\right\|_\infty^{\frac{\beta}{\alpha}}\E\left[\sup\limits_{s\in[0,T]} |g^2_s|^{\beta} \right],
$$
$$
\E\left[\left(\int_0^T \gamma_s |z^{n-1}_s|^\alpha \dd s\right)^{\frac{\beta}{\alpha}}\right] \leq \left\|\int_0^T \gamma_s^{\frac{2}{2-\alpha}}\dd s\right\|_\infty^{\frac{\beta(2-\alpha)}{2\alpha}}\E\left[\left(\Dis \int_0^T |z^{n-1}_s|^2 \dd s\right)^{\frac{\beta}{2}}\right],
$$
and $z^{n+i-1}_\cdot$ has a similar estimate. In addition,
$$
\E\left[\left(\int_0^T \gamma_s|y^{n}_s|^\alpha \dd s\right)^{\frac{\beta}{\alpha}}\right] \leq \left\|\int_0^T \gamma_s \dd s\right\|_\infty^{\frac{\beta}{\alpha}}\E\left[\supt|y^{n}_s|^\beta\right].\vspace{0.1cm}
$$
Note that $y^{n}_\cdot$ belongs to $\s^\beta$, $z^{n-1}_\cdot$ and $z^{n+i-1}_\cdot$ belong to $\M^\beta$, $\E\big[\int_0^Tg^1_t\ {\rm d}t\big]<+\infty$ and $\E\big[\supt |g^2_t|\big]<+\infty$. By the previous inequalities along with \eqref{eq:3.1} and \eqref{eq:3.3} we have for  $\beta\in (\alpha,1)$,
\begin{equation}\label{107}
G^{n,i}(\cdot)\in \s^\frac{\beta}{\alpha}.
\end{equation}
On the other hand, it follows from \eqref{eq:4.1} and (\ref{105}) that
$$
|\hat{y}^{n,i}_t|\leq\E\left[\eta_{n,i}+\left.\int_t^T Au_s |\hat{y}^{n,i}_s|\ \mathrm{d}s\right|\F_t\right], \ \ t\in[0,T],
$$
where
$$
\eta_{n,i}:=2\int^{T}_{0}\gamma_s\left(g^1_s+g^2_s+|y^n_s|+|z^{n-1}_s|+|z^{n+i-1}_s|\right)^\alpha\ \mathrm{d}s+A\int_0^T u_s \dd s.\vspace{0.1cm}
$$
In light of \eqref{eq:4.24}, applying the stochastic Gronwall-type inequality (see \cref{pro:2.3}) yields that
\begin{equation}\label{110}
 |\hat{y}^{n,i}_t| \leq e^{A\|\int_0^Tu_s \dd s\|_\infty}\left(A\left\|\int_0^T u_s\dd s\right\|_\infty+G^{n,i}(t)\right)\leq e^{AM}\left( AM+G^{n,i}(t)\right), \ \ t\in[0,T].
\end{equation}
Combining (\ref{107}) and (\ref{110}), we can draw the conclusion that for each $n,i\geq 1$ and $\beta\in (\alpha,1)$,
$$
|\hat y_\cdot^{n,i}|\in \s^\frac{\beta}{\alpha}.\vspace{0.2cm}
$$

\textbf{Second step:} \ \ We complete the proof by distinguishing two cases: $\alpha\in [0,1/2)$ and $\alpha\in [1/2,1)$. Let us first roughly explain why these two cases are fundamentally different. \vspace{0.2cm}

For the case of $\alpha\in [0,1/2)$, we can pick $\beta:=2\alpha$ such that $\hat y_\cdot^{n,i}\in \s^2$, and then use \cref{pro:4.6} along with \cref{pro:2.3,pro:2.4} to prove that $\{(\hat y_\cdot^{n,1},\hat z_\cdot^{n,1})\}$ is a cauchy sequence in $\s^2\times\M^2$, which converges to a process $(Y_\cdot,Z_\cdot)\in \s^2\times\M^2$, and thus $(Y_\cdot+y^1_\cdot,Z_\cdot+z^1_\cdot)$ is the desired $L^1$ solution. In this case, as soon as the generator $g$ satisfies (H1), \cref{pro:4.6} can be utilized since $g$ must satisfy (H1a)$_2$ via \cref{pro:4.1}. \vspace{0.2cm}

However, for the case of $\alpha\in [1/2,1)$, it is evident that $\beta/\alpha\in (1,2)$ for each $\beta\in (\alpha,1)$, and thus we only know that $\hat y_\cdot^{n,i}\in \s^p$ for some $p\in (1,2)$. Consequently, in this case we have to prove $\{(\hat y_\cdot^{n,1},\hat z_\cdot^{n,1})\}$ is a cauchy sequence in $\s^p\times\M^p$, and then the use of \cref{pro:4.7} is necessary. In order to apply \cref{pro:4.7}, the generator $g$ is required to satisfy (H1b)$_p$. In light of \cref{pro:4.1}, this is the reason that when $\alpha \in [1/2,1)$, we need to additionally assume in (ii) of assumption (H4) that \eqref{eq:3.2} holds for some $\bar p\in 1$.\vspace{0.2cm}

{\bf Case (i)}: $\alpha \in [1/2,1)$. By taking $\beta:=\alpha p$ with $p\in (1,\frac{1}{\alpha}\wedge\bar p)\subset(1,2)$, where $\bar{p}>1$ is defined in (ii) of assumption (H4), we have for each $n,i\geq 1$,
\begin{equation}\label{eq:4.30}
\hat y_\cdot^{n,i}\in \s^p(0,T;\R^k).
\end{equation}
By assumptions (H1) and (ii) of (H4) along with (iii) of \cref{rm3.1} and \cref{pro:4.1}, we can deduce that the generator $g$ satisfies the assumption (H1a)$_2$ with a linearly growing function $\kappa(\cdot)\in {\bf S}$ and the assumption (H1b)$_p$ with a linearly growing function $\varrho(\cdot)\in {\bf S}$, which along with (H3), (i) of (H4) and \eqref{eq:4.1} yields that for each $n,i\geq 1$, the generator $\bar{g}^{n,i}$ satisfies the assumption (A1) with
$$
\mu_\cdot:=Au_\cdot, \ \ \lambda_\cdot:\equiv0, \ \ f_\cdot:=2\gamma_\cdot\, (g^1_\cdot+g^2_\cdot+|y^n_\cdot|+|z^{n-1}_\cdot|+|z^{n+i-1}_\cdot|)^\alpha, \ \ \varphi_\cdot:=Au_\cdot,
$$
and the assumption (A3) with
$$
\mu_\cdot:=u_\cdot, \ \ \phi(\cdot):=\varrho(\cdot), \ \ \lambda_\cdot:\equiv0, \ \ f_\cdot:=v_\cdot\, |\hat z^{n-1,i}_\cdot|.
$$
Then, by \cref{pro:4.2,pro:4.7} along with \eqref{107}, \eqref{eq:4.30} and H\"{o}lder's inequality we can deduce that for each $n\geq 2$ and $i\geq1$, $\hat{z}_\cdot^{n,i}\in\M^p(0,T;\R^{k\times d})$ and there exists a constant $C_{p,M}>0$ depending only on $(p,M)$ such that for each $t\in\T$ and each pair of stopping times $\sigma_1$ and $\sigma_2$ satisfying $0\leq\sigma_1\leq\sigma_2\leq T$,
\begin{equation}\label{3.26}
\begin{array}{l}
\Dis \E\left[\sup\limits_{s\in [t,T]}|\tilde
		y^{n,i}_s|^p\bigg|\F_t\right]+\E\left[\left(\int_t^T |\tilde z^{n,i}_s|^2 \dd s\right)^{p\over 2}\bigg|\F_t\right]\vspace{0.2cm}\\
\ \ \leq \Dis C_{p,M}\E\left[|\tilde y^{n,i}_{\sigma_2}|^p+\int_t^Tu_s\varrho\left(|\tilde y^{n,i}_s|^p \right)\dd s+\left(\int_t^T \tilde v_s^2 \dd s\right)^{p\over 2}\left(\int_t^T |\tilde z^{n-1,i}_s|^2 \dd s\right)^{p\over 2}\bigg|\F_t\right],
	\end{array}
\end{equation}
where
$$
\tilde{y}_t^{n,i}:=\mathbbm{1}_{\sigma_1\leq t }\hat{y}_{t\wedge \sigma_2}^{n,i}, \ \ \ \tilde{z}_t^{n,i}:=\mathbbm{1}_{\sigma_1\leq t \leq\sigma_2}\hat{z}_t^{n,i} \ \ \ \text{and}\ \ \  \tilde{v}_t:=\mathbbm{1}_{\sigma_1\leq s \leq\sigma_2}v_t.\vspace{0.2cm}	
$$
Now, let us fix a positive integer $N$ such that
\begin{equation}\label{24}
\frac{K}{N} \leq \frac{1}{4\bar C},
\end{equation}
where
$$
K:=\left\|\int_0^T u_s \dd s\right\|_\infty+\left\|\int_0^T v_s^2 \dd s\right\|_\infty^{p\over 2}\leq M+M^{\frac{p}{2}}
$$
and
$$
\bar C:=C_{p,M}e^{AC_{p,M}\|\int_0^Tu_s\dd s\|_\infty}\leq C_{p,M}e^{AMC_{p,M}}.\vspace{0.2cm}
$$
In the sequel, by virtue of $(\F_t)$-measurable stopping times we divide the time interval $[0,T]$ into some small subintervals $[T_{j-1},T_j],\  j=1,2,\ldots N$ as follows: $T_{0}:=0$,
$$
\begin{aligned}
&T_{1}:=\inf \left\{t \geq 0: \left(\int_{0}^{t}v_s^{2}  \dd s\right)^{p\over 2} \geq \frac{K}{N}\right\} \wedge T, \\
&\quad \vdots 
\end{aligned}
$$
$$
\begin{aligned}
&T_{j}:=\inf \left\{t \geq T_{j-1}: \left(\int_{0}^{t}v_s^{2}  \dd s\right)^{p\over 2} \geq \frac{j K}{N}\right\} \wedge T, \\
&\quad \vdots \\
&T_{N}:=\inf \left\{t \geq T_{N-1}: \left(\int_{0}^{t}v_s^{2}  \dd s\right)^{p\over 2} \geq \frac{N K}{N}\right\} \wedge T=T.\vspace{0.2cm}
\end{aligned}
$$
Thus, for each $[T_{j-1},T_{j}]\subset[0,T]$ , $j=1,2,\ldots N$, we have, in view of (\ref{24}),
\begin{equation}\label{25}
\bar C\left(\int_0^T \mathbbm{1}_{T_{j-1}\leq s \leq T_{j}}v_s^{2} \dd s\right)^{p\over 2}\leq\frac{\bar CK}{N}\leq\frac{1}{4}.
\end{equation}
Here, we mention that the previous choice of the constant $N$ and the stopping times $T_j$ ensures that after using the stochastic Gronwall-type inequality for \eqref{3.26} on each stopping time interval $[T_{j-1},T_j]$, the coefficient before the key term is smaller than $1/4$ such that the subsequent induction can proceed smoothly.

Now, let $\sigma_1=T_{N-1}, \sigma_2=T_{N}=T$ in (\ref{3.26}). In light of $\hat y^{n,i}_T=0,\ \hat y_\cdot^{n,i}\in\s^p$ and $\hat z_\cdot^{n-1,i}\in\M^p$ along with (\ref{3.26}) and (\ref{25}), by the stochastic Gronwall-type inequality (see \cref{pro:2.3}) we can deduce that for each $n\geq 2,\ i\geq 1$ and $t\in\T$,
\begin{equation}\label{34}
\begin{array}{l}
\E\left[\sup\limits_{s\in[t,T]}|\tilde{y}_s^{n,i}|^p+\left(\Dis \int_t^T|\tilde{z}_s^{n,i} |^2 \dd s\right)^{p\over 2}\bigg|\F_t\right]\\
\ \ \leq \bar CA\left\|\Dis \int_0^Tu_s\dd s\right\|_\infty+\bar C\E\left[\left(\Dis \int_0^T \tilde v^2_s\dd s\right)^{p\over 2}\left(\Dis \int_0^T |\tilde{z}^{n-1,i}_s|^2\dd s\right)^{p\over 2}\bigg|\F_t\right]\vspace{0.1cm}\\
\ \ \leq \bar CAK + \Dis \frac{1}{4}\E\left[\left(\Dis \int_0^T |\tilde{z}^{n-1,i}_s|^2\dd s\right)^{p\over 2}\bigg|\F_t\right].
\end{array}
\end{equation}
Furthermore, by taking supremum with respect to $i$ and $n$ on both sides of (\ref{34}) and using a induction similar to that in \citet[p.1877-1878,1880]{Fan18}, we can derive the following inequality
\begin{equation}\label{26}
\begin{array}{lll}
\sup\limits_{n\geq2}\supi|
 \tilde y^{n,i}_t|^p & \leq & \sup\limits_{n\geq2}\supi\E\left[\sup\limits_{s\in [t,T]}|
 \tilde y^{n,i}_s|^p+\left(\Dis \int_t^T |\tilde{z}^{n,i}_s|^2\dd s\right)^{p\over 2} \dd s\bigg|\F_t \right]\\
& \leq & \Dis \bar CAK+\frac{1}{4}\sup_{n\geq 2}\sup_{i\geq 1}\E\left[\left(\Dis \int_0^T |\tilde{z}^{n-1,i}_s|^2\dd s\right)^{p\over 2}\bigg|\F_t\right]\vspace{0.1cm}\\
& \leq & \Dis 3\bar CAK+\E\left[\left(\int_0^T |z^2_s-z^1_s|^2\dd s\right)^{p\over 2}\bigg|\F_t\right]:=J_t,\ \ \ t\in[0,T].
\end{array}
\end{equation}
It then follows from Fubini's theorem that\vspace{0.1cm}
\begin{equation}\label{35}
\begin{array}{l}
\E\left[\Dis \int_0^T u_s\, \sup\limits_{n\geq2}\supi|
 \tilde y^{n,i}_s|^p\dd s\right]\leq \E\left[\Dis \int_0^T u_s J_s\dd s\right]\vspace{0.2cm}\\
\ \ \leq 3\bar CAK\left\|\Dis \int_0^Tu_s \dd s \right\|_{\infty}+\Dis\E\left[\Dis \int_0^Tu_s\E\left[\Dis \left(\int_0^T |z^2_s-z^1_t|^2\dd t\right)^{p\over 2}\bigg|\F_s\right]\dd s\right]\vspace{0.2cm}\\
\ \ \leq 3\bar CK^2A + \Dis K\E\left[\left(\Dis \int_0^T |z^2_s-z^1_s|^2\dd s\right)^{p\over 2}\right]<+\infty.
\end{array}
\end{equation}
Thus, taking first supremum with respect to $i$ then limit superior with respect to $n$ on both sides of (\ref{3.26}), by (\ref{25}), Fatou's lemma and the monotonicity and continuity of $\varrho(\cdot)$ we can deduce that
$$
\begin{array}{lll}
\varlimsup\limits_{n\rightarrow\infty}\supi|\tilde y_t^{n,i}|^p
&\leq &\Dis \varlimsup\limits_{n\rightarrow\infty}\supi\E\left[\sup\limits_{s\in[t,T]}|\tilde y_s^{n,i}|^p+\Dis \frac{3}{4}\left(\int_t^T|\tilde z_s^{n,i}|^2 \dd s\right)^{p\over 2}\bigg|\F_t\right]\vspace{0.1cm}\\
&\leq & \Dis C_{p,M}\E\left[\int_t^T u_s\varrho\left(\varlimsup\limits_{n\rightarrow\infty}\supi|\tilde y^{n,i}_s|^p \right)\dd s\bigg|\F_t\right],\ \ \ t\in\T.\vspace{0.1cm}
\end{array}
$$
And, in light of \eqref{35}, by the stochastic Bihari-type inequality (see \cref{pro:2.4}) we get
\begin{equation}\label{37}
\varlimsup\limits_{n\rightarrow\infty}\supi|\tilde y_t^{n,i}|^p=0, \ \ t\in[0,T].
\end{equation}
Furthermore, in light of (\ref{35}), (\ref{37}) and the fact that $\varrho(\cdot)$ is a continuity function with $\varrho(0)=0$, by Lebesgue's dominated convergence theorem we can obtain that for each $t\in\T$,
$$
\begin{array}{l}
\varlimsup\limits_{n\rightarrow\infty}\supi\E\left[\sup\limits_{s\in[t,T]}|\tilde y_s^{n,i}|^p+\Dis \frac{3}{4}\left(\Dis \int_t^T|\tilde z_s^{n,i}|^2 \dd s\right)^{p\over 2}\right]\vspace{0.1cm}\\
\ \ \leq \Dis C_{p,M}\E\left[\int_t^T u_s\varrho\left(\varlimsup\limits_{n\rightarrow\infty}\supi|\tilde y^{n,i}_s|^p \right)\dd s\right]=0,
\end{array}
$$
which means that
\begin{equation}\label{3.35}
	\lim_{n\rightarrow\infty}\supi\E\left[\sup\limits_{s\in[t,T]}|\hat y_{s\wedge T}^{n,i}\mathbbm{1}_{T_{N-1}\leq s}|^p+\left(\Dis \int_t^T|\hat z_s^{n,i}|^2\mathbbm{1}_{T_{N-1}\leq s \leq T} \dd s\right)^{p\over 2}\right]=0.\vspace{0.1cm}
\end{equation}

In the sequel, let $\sigma_1=T_{N-2}, \sigma_2=T_{N-1}$ in (\ref{3.26}). In light of (\ref{25}) and (\ref{26}), by applying the stochastic Gronwall-type inequality to (\ref{3.26}) we can deduce that for each $n\geq2$, $i\geq1$ and $t\in[0,T]$,
$$
\E\left[\sup\limits_{s\in[t,T]}|\tilde{y}_s^{n,i}|^p+\left(\Dis \int_t^T|\tilde{z}_s^{n,i} |^2 \dd s\right)^{p\over 2}\bigg|\F_t\right] \leq  \bar C(AK+J_t)+ \Dis \frac{1}{4}\E\left[\left(\Dis \int_0^T |\tilde{z}^{n-1,i}_s|^2\dd s\right)^{p\over 2}\bigg|\F_t\right],
$$
where
$$
\tilde{y}_t^{n,i}:=\mathbbm{1}_{T_{N-2}\leq t }\hat{y}_{t\wedge T_{N-1}}^{n,i},\ \  \tilde{z}_t^{n,i}:=\mathbbm{1}_{T_{N-2}\leq t \leq T_{N-1}}\hat{z}_t^{n,i}\ \  \text{and}\ \ \tilde{v}_t:=\mathbbm{1}_{T_{N-2}\leq t \leq T_{N-1}}v_t.\vspace{0.1cm}
$$
Thus, with the help of (\ref{3.35}), an induction similar to (\ref{26})-(\ref{3.35}) yields that for each $t\in[0,T]$,
$$
\lim_{n\rightarrow\infty}\supi\E\left[\sup\limits_{s\in[t,T]}|\hat y_{s\wedge T_{N-1}}^{n,i}\mathbbm{1}_{T_{N-2}\leq t}|^p+\left(\Dis \int_t^T|\hat z_s^{n,i}|^2\mathbbm{1}_{T_{N-2}\leq t \leq T_{N-1}} \dd s\right)^{p\over 2}\right]=0.\vspace{0.1cm}
$$
Furthermore, by proceeding the above process we can verify that for any $j=1,...,N$ and $t\in[0,T]$,
$$
\lim_{n\rightarrow\infty}\supi\E\left[\sup\limits_{s\in[t,T]}|\hat y_{s\wedge T_{N-j}}^{n,i}\mathbbm{1}_{T_{N-j+1}\leq t}|^p+\left(\Dis \int_t^T|\hat z_s^{n,i}|^2\mathbbm{1}_{T_{N-j+1}\leq t \leq T_{N-j}} \dd s\right)^{p\over 2}\right]=0,
$$
which means that $(y^n_t-y^1_t, z^n_t-z^1_t)_{t\in\T}$ is a Cauchy sequence in ${\s}^p(0,T;\R^k)\times {\M}^p(0,T;\R^{k\times d})$. By  $(Y_\cdot,Z_\cdot)$ we denote the limit of $(y^n_\cdot-y^1_\cdot,z^n_\cdot-z^1_\cdot)$ in the space ${\s}^p\times {\M}^p$, then
\begin{equation}\label{37*}
\limn\E\left[\sup_{t\in [0,T]}|(y^n_t-y^1_t)-Y_t|^p+\left(\int_{0}^T
|(z^n_t-z^1_t)-Z_t|^2 \dd t\right)^{\frac{p}{2}}\right]=0.	
\end{equation}
And, since $\E\left[\int_0^T u_t|y_t^1|\dd t\right]<+\infty$ and $Y_\cdot\in{\s}^p$, we have
\begin{equation}\label{30}
\E\left[\int_0^Tu_t|y_t^1+Y_t|
\dd t\right]
\leq\E\left[\int_0^Tu_t|y_t^1|\dd t\right]+M\E\left[\supt|Y_t|\right]<+\infty.\vspace{0.1cm}
\end{equation}
Finally, note that $(y^1_\cdot,z^1_\cdot)\in \s^\beta\times\M^\beta$ for each $\beta\in (0,1)$,  $y^1_\cdot$ belongs to class $(D)$ and further satisfies \eqref{eq:3.4} and \eqref{eq:3.5}. In light of (\ref{37*}), (\ref{30}) and assumptions (H2) and (H3), by taking limit with respect to $n$ on both sides of BSDE (\ref{18}) and using Lebesgue's dominated convergence theorem we conclude that $(y_t,z_t)_{t\in[0,T]}:=(Y_t+y^1_t,Z_t+z^1_t)_{t\in[0,T]}$ is the $L^1$ solution to BSDE $(\xi,T,g)$ satisfying \eqref{eq:3.4} and \eqref{eq:3.5}\vspace{0.3cm}.

{\bf Case (ii)}: $\alpha\in[0,1/2)$. Take $\beta=2\alpha\in (0,1)$, then for each $n,i\geq1$, we have
\begin{equation}\label{32}
	\hat y_\cdot^{n,i}\in \s^2(0,T;\R^k).
\end{equation}
By (H1) and \cref{pro:4.1} we know that the generator $g$ satisfies (H1a)$_2$ with a linearly growing function $\kappa(\cdot)\in\s$. Then, by virtue of assumptions (H1a)$_2$ and (i) of (H4) along with \eqref{eq:4.1} and \eqref{107}, it is not hard to verify that the generator $\bar{g}^{n,i}$ satisfies the assumption (A1) with
$$p:=2, \ \ \mu_\cdot:=Au_\cdot, \ f_\cdot:=2\gamma_\cdot\, (g^1_\cdot+g^2_\cdot+|y^n_\cdot|
+|z^{n-1}_\cdot|+|z^{n+i-1}_\cdot|)^\alpha, \ \ \lambda_\cdot:\equiv0, \ \ \varphi_\cdot:=Au_\cdot.$$
It then follows from \cref{pro:4.2} along with \eqref{32} and \eqref{107} that $\hat z_\cdot^{n,i}\in \M^2(0,T;\R^{k\times d})$ such that $(\hat y_\cdot^{n,i},\hat z_\cdot^{n,i})$ is an $L^2$ solution of BSDE (\ref{33}). On the other hand, by (H1a)$_2$, (H3) and the definition of $\bar g^{n,i}$, it can also be verified that the generator $\bar g^{n,i}$ satisfies the assumption (A2) with
$$
p:=2, \ \ \mu_\cdot:=u_\cdot, \ \ \psi(\cdot):=\kappa(\cdot), \ \ \lambda_\cdot:\equiv 0, \  \ f_\cdot:=v_\cdot|\hat z^{n-1,i}_\cdot|.
$$
Thus, by \cref{pro:4.6} and H\"{o}lder's inequality we can deduce that there exists a constant $C_M>0$ depending only on $M$ such that for each $n\geq 2$, $i\geq1$, $t\in\T$ and any pair of $(\F_t)$-measurable stopping times $\sigma_1$ and $\sigma_2$ satisfying $0\leq\sigma_1\leq\sigma_2\leq T$, we have
\begin{equation}\label{23}
\begin{array}{l}
\E\left[\sup\limits_{s\in[t,T]}|\tilde{y}_s^{n,i}|^2\bigg|\F_t\right]+\E\left[\Dis \int_t^T|\tilde{z}_s^{n,i} |^2 \dd s\bigg|\F_t\right]\vspace{0.2cm}\\
\ \ \leq \Dis C_M\E\left[|\tilde{y}_ {\sigma_2}^{n,i}|^2+
\int_t^T u_s\kappa\left(|\tilde{y}^{n,i}_s|^2 \right)\dd s+\int_0^T \tilde{v}_s^2 \dd s \int_t^T |\tilde{z}^{n-1,i}_s|^2\dd s\bigg|\F_t\right],\vspace{-0.1cm}
\end{array}
\end{equation}
where
$$
\tilde{y}_t^{n,i}:=\mathbbm{1}_{\sigma_1\leq t }\hat{y}_{t\wedge \sigma_2}^{n,i},\ \ \ \tilde{z}_t^{n,i}:=\mathbbm{1}_{\sigma_1\leq t \leq\sigma_2}\hat{z}_t^{n,i}\ \ \  {\text and}\ \ \  \tilde{v}_t:=\mathbbm{1}_{\sigma_1\leq s \leq\sigma_2}v_t.\vspace{0.1cm}
$$
With (\ref{23}) in hands and in light of \eqref{eq:4.1}, by an identical argument with (\ref{24})-(\ref{30}), we can verify that BSDE $(\xi,T,g)$ admits an $L^1$ solution $(y_t,z_t)_{t\in\T}$ satisfying \eqref{eq:3.4} and \eqref{eq:3.5}. The proof of the existence part of \cref{th} is then complete.
\end{proof}


\subsection{Proof of the uniqueness part of \cref{th}\vspace{0.1cm}}

In this subsection, we will prove the following \cref{pro:4.8} whose a direct consequence is the uniqueness part of \cref{th}. Note that the condition (ii) of (H4) is not required in \cref{pro:4.8}. The whole idea of the proof is mainly lent from Theorem 1 in \citet{Fan18}. The main difference is that the additional condition (\ref{eq:3.4}) is required in \cref{pro:4.8} for the uniqueness of the $L^1$ solution of BSDE \eqref{101}. This condition enable us to apply smoothly the stochastic Gronwall-type inequality.

\begin{pro}\label{pro:4.8}
Let assumptions (H1), (H3), (i) of (H4) and (H5) hold. Then, BSDE \eqref{101} admits at most one solution $(y_\cdot,z_\cdot)\in \bigcup_{\beta\in(\alpha,1)}\s^\beta\times\M^\beta$ such that $y_\cdot$ belongs to class $(D)$ and satisfies (\ref{eq:3.4}).
\end{pro}

\begin{proof}
Assume that $(y_t,z_t)_{t\in[0,T]}$ and $(y'_t,z'_t)_{t\in[0,T]}$ are both solutions of BSDE \eqref{101}, $(y_t,z_t)_{t\in[0,T]}$ and $(y'_t,z'_t)_{t\in[0,T]}$ belong to $\s^\beta\times\M^\beta$ for some $\beta\in(\alpha,1)$, $(y_t)_{t\in[0,T]}$ and $(y'_t)_{t\in[0,T]}$ belong to class $(D)$, and
\begin{equation}\label{eq:4.43}
\E\left[\int_0^T u_t(|y_t|+|y'_t|)\dif t\right]<+\infty.
\end{equation}
It is obvious that $(\hat y_t, \hat z_t)_{t\in[0,T]}:=(y_t-y'_t, z_t-z'_t)_{t\in[0,T]} $ is an adapted solution of the following BSDE:
\begin{equation}\label{109}
\hat{y_t}=\int^T_t\hat g(s,\hat y_s,\hat z_s)\ \mathrm{d} s-\int^T_t\hat z_s\,\mathrm{d} B_s,\quad t\in[0,T],
\end{equation}
where $\as$, for each $(y,z)\in\R^k\times\R^{k\times d}$,\vspace{0.1cm}
$$\hat g(t,y,z):=g(t,y+y'_t,z+z'_t)-g(t,y'_t,z'_t).\vspace{0.1cm}$$
Then, in light of (H1) and (i) of (H4), by a argument similar to (\ref{33})-(\ref{107}) in the proof of the existence part of Theorem \ref{th}, it is not hard to verify that for each $t\in\T$,
$$
|\hat{y}_{t}|\leq\E\left[
	\left.\int^{T}_{t} u_s\rho(|\hat{y}_s|)\ \mathrm{d}s\right|\F_t\right]+G(t),
$$
where
$$
G(t):=2\E\left[\left.\int^{T}_{0}\gamma_s\left(g^1_s+g^2_s+
|y'_s|+|z_s|+|z'_s|\right)^\alpha\ \mathrm{d}s\right|\F_t\right]\in\s^{q}\ \  \text{with}\ \ q:=\frac{\beta}{\alpha}>1.\vspace{0.1cm}
$$
In light of \eqref{eq:4.1} and \eqref{eq:4.43}, by the stochastic Gronwall-type inequality (see \cref{pro:2.3}) we deduce that for each $t\in[0,T]$,
$$
|\hat{y}_t| \leq e^{A\|\int_0^Tu_s \dd s\|_\infty}\left(A\|\int_0^T u_s\dd s\|_\infty+G(t)\right)\leq e^{AM}\left( AM+G(t)\right),
$$
which yields that
\begin{equation}\label{108}
\hat y_\cdot=(y_\cdot-y'_\cdot)\in \s^{q}.\vspace{0.1cm}
\end{equation}
Furthermore, by (H1) and \cref{pro:4.1} we know that the generator $g$ satisfies (H1a)$_2$ with a linearly growing function $\kappa(\cdot)\in {\bf S}$. Then, in light of (H1a)$_2$ and (H3) of the generator $g$ along with \eqref{eq:4.1}, we deduce that the generator $\hat{g}$ satisfies the assumption (A1) with
$$p:=q, \ \ \mu_\cdot:=Au_\cdot, \ \ \lambda_\cdot:=v_\cdot,\ \ f_\cdot:\equiv 0,\ \ \varphi_\cdot:=Au_\cdot.\vspace{0.1cm}$$
It then follows from \cref{pro:4.2} that $\hat z_\cdot\in\M^q$ and then $(\hat y_t, \hat z_t)_{t\in[0,T]}$ is an $L^q$ solution of BSDE (\ref{109}). On the other hand, by (H1) and \cref{pro:4.1} we know that the generator $g$ satisfies assumption (H1a)$_q$ with a linearly growing function $\bar \kappa(\cdot)\in {\bf S}$. Then, in light of (H1a)$_q$ and (H3) along with the definition of $\hat{g}$, it can also be verified that the generator $\hat{g}$ satisfies the assumption (A2) with
$$p:=q, \ \ \mu_\cdot:=u_\cdot,\ \ \psi(\cdot):=\bar\kappa(\cdot),\ \ \lambda_\cdot:=v_\cdot,\ \ f_\cdot:\equiv 0.\vspace{0.1cm}$$
It then follows from \cref{pro:4.5} with $u=t$ that there exists a constant $C_{q,M}>0$ depending only on $q$ and $M$ such that for each $t\in\T$,
\begin{equation}\label{112}
|\hat y_t|^q\leq\E\left[\sup_{s\in [t,T]}|\hat y_s|^q\bigg|\F_t\right]\leq C_{q,M}\E\left[\int_t^T u_s\bar\kappa\left(|\hat y_s|^q\right)\ {\rm d}s\bigg|\F_t\right].\vspace{0.1cm}
\end{equation}
In consideration of (\ref{108}) and (\ref{112}), the stochastic Bihari-type inequality (see \cref{pro:2.4}) yields that for each $t\in\T$,
\begin{equation}\label{113}
	|\hat y_t|=0.
\end{equation}
Furthermore, in light of (H1a)$_2$ and (H3) of the generator $g$ along with (\ref{115}), we can check that for each $m\geq 1$, the generator $\hat{g}$ of BSDE (\ref{109}) satisfies the assumption (A1) with
$$
p:=q, \ \ \mu_\cdot:=(m+A)u_\cdot, \ \ \lambda_\cdot:=v_\cdot, \ \ f_\cdot:\equiv 0, \ \ \varphi_\cdot:=\kappa\left(\frac{A}{m}\right)
u_\cdot.
$$
It then follows from \cref{pro:4.2} with $u=t=0$ that there exists a constant $C_q>0$ depending only on $q$ such that for each $m\geq1$,
\begin{equation}\label{114*}
\E\left[\left(\int_0^T |\hat z_t|^2\ {\rm
d}t\right)^{\frac{q}{2}}\right]\leq C_{q}[1+(m+A)M]^{\frac{q}{2}}\E\left[\sup\limits_{t\in [0,T]}|\hat y_t|^q\right]+C_{q}\E\left[\left(\int_0^T \kappa\left({A\over m}\right)u_t \dd t\right)^{\frac{q}{2}}\right].
\end{equation}
In light of (\ref{113}) and the fact that $\kappa(\cdot)$ is a continuous function with $\kappa(0)=0$, by sending $m\rightarrow\infty$ in the last inequality \eqref{114*} we obtain
$$
\begin{array}{lll}
\E\left[\left(\Dis \int_0^T |\hat z_s|^2\ {\rm
d}s\right)^{\frac{q}{2}}\right]=\E\left[\left(\Dis \int_0^T |z_s-z'_s|^2\
{\rm d}s\right)^{\frac{q}{2}}\right]=0.
\end{array}
$$
The proof is complete.
\end{proof}

\appendix
\section*{Appendix: The proof of \cref{pro:4.7}}
\renewcommand{\appendixname}{}
\renewcommand{\theequation}{A.\arabic{equation}}
\setcounter{equation}{0}

In this section, we give the proof of \cref{pro:4.7}. First of all, for each $n\geq1$, introduce the following $(\F_t)$-measurable stopping time:
$$\tau_{n}:=\inf \left\{t \in[0, T]: \int_{0}^{t}|z_{s}|^{2} {\rm d}s \geq n\right\} \wedge T.$$
Applying It\^{o}'s formula to $\left|y_{t}\right|^{2}$ yields that for each $n \geq 1$ and $t\in [0, T]$,
\begin{align}\label{eq:4.21}
\begin{split}
\left|y_{t \wedge \tau_{n}}\right|^{2}+\int_{t \wedge \tau_{n}}^{\tau_{n}}\left|z_{s}\right|^{2}{\rm d}s
=\left|y_{\tau_{n}}\right|^{2}+2 \int_{t \wedge \tau_{n}}^{\tau_{n}}\left\langle y_{s}, g\left(s, y_{s}, z_{s}\right)\right\rangle {\rm d}s-2 \int_{t \wedge \tau_{n}}^{\tau_{n}}\left\langle y_{s}, z_{s} {\rm d} B_{s}\right\rangle.
\end{split}
\end{align}
It follows from assumption (A3) that $\as$,
\begin{align}\label{eq:4.22}
\begin{split}
2\left\langle y_{t},\ g\left(t, y_{t}, z_{t}\right)\right\rangle \leq 2|y_t|\mu_t \phi^{{1\over p}}(|y_t|^p)+2\lambda_t^2 |y_t|^{2}+\frac{\left|z_{t}\right|^{2}}{2}+2\left|y_{t}\right| f_{t}.
\end{split}
\end{align}
By \eqref{eq:4.21} and \eqref{eq:4.22} we obtain that there exists a constant $c_{p}>0$ depending only on $p$ such that
\begin{align}\label{eq:4.23}
\begin{split}
\left(\int_{t\wedge\tau_n}^{\tau_n}|z_s|^2{\rm d}s\right)^\frac{p}{2}\leq& c_p\left(1+\int_{t\wedge\tau_n}^{\tau_n} \lambda_s^2\  {\rm d}s\right)^\frac{p}{2}\sup_{s\in[t\wedge\tau_n,\tau_n]}|y_s|^p
+c_p\left(\int_{t\wedge\tau_n}^{\tau_n}\mu_s \phi^{{1\over p}}(|y_s|^p) {\rm d}s\right)^p\\
&+c_p\left(\int_{t\wedge\tau_n}^{\tau_n}f_s{\rm d}s\right)^p+c_p\left|\int_{t\wedge\tau_n}^{\tau_n}\left<y_s,z_s{\rm d}B_s\right>\right|^\frac{p}{2}, \ \ \ t\in[0,T].
\end{split}
\end{align}
Furthermore, the Burkholder-Davis-Gundy (BDG) inequality yields that there exists a constant $d_p>0$ depending only on $p$ such that for each $0\leq u \leq t\leq T$,
\begin{align*}
\begin{split}
c_p\E\left[\left|\int_{t\wedge\tau_n}^{\tau_n}\left<y_s,z_s{\rm d}B_s\right>\right|^\frac{p}{2}\bigg|\F_u\right]&\leq d_p\E\left[\left(\int_{t\wedge\tau_n}^{\tau_n}|y_s|^2|z_s|^2{\rm d}s\right)^{\frac{p}{4}}\bigg|\F_u\right]\\
&\leq{\frac{{d^2_p}}{2}}\E\left[\sup_{s\in[t\wedge\tau_n,\tau_n]}
|y_s|^p\bigg|\F_u\right]+\frac{1}{2}
\E\left[\left(\int_{t\wedge\tau_n}^{\tau_n}|z_s|^2{\rm d}s\right)^{\frac{p}{2}}\bigg|\F_u\right].
\end{split}
\end{align*}
Thus, in light of the last inequality, taking the conditional mathematical expectation with respect to $\F_u$ in both sides of \eqref{eq:4.23} and letting $n\rightarrow +\infty$, by Fatou's lemma and Lebesgue's dominated convergence theorem we can conclude that $z_\cdot\in\M^p(0,T;\R^{k\times d})$ and there exists a constant $\bar C_{p,M}>0$ depending only on $(p,M)$ such that for each $0\leq u \leq t\leq T$,
\begin{align}\label{4.23-1}
\begin{split}
&\E\left[\left(\Dis \int_t^T
|z_s|^2 {\rm d}s\right)^\frac{p}{2}\bigg|\F_u\right]\\
&\ \ \leq \bar C_{p,M} \mathbb{E}\left[\sup\limits_{s\in[t,T]}\left|y_{s}\right|^{p}
+\left(\int_{t}^{T}\mu_{s}\phi^{{1\over p}}(|y_s|^p) \mathrm{d}s\right)^p + \left(\int_{t}^{T} f_{s}\mathrm{d}s\right)^{p}\bigg|{\F}_{u}\right].
\end{split}
\end{align}

On the other hand, in light of (A3), by an identical way with Lemma 2.6 in \citet{LiFan2023} we can prove that there exists a constant $\tilde C_{p,M}>0$ depending only on $p$ and $M$ such that for each $0\leq u \leq t\leq T$,
\begin{align}\label{4.23-2}
\mathbb{E}\left[\sup _{s \in[t, T]}\left|y_{s}\right|^{p} \bigg|{\F}_{u}\right] \leq \tilde C_{p, M} \mathbb{E}\left[|\xi|^{p}+\int_{t}^{T} \mu_{s} |y_s|^{p-1}\phi^{{1\over p}}(|y_s|^p) \mathrm{d} s+\left(\int_{t}^{T} f_{s} \mathrm{ d} s\right)^{p} \bigg|{\F}_{u}\right].
\end{align}
In fact, according to Corollary 2.3 in \cite{BriHu06} and assumption (A2), we know that
\begin{align}\label{2.4}
\begin{split}
&|y_t|^p+c(p)\int_{t}^{T}|y_s|^{p-2}\mathbbm{1}_{|y_s|\neq 0}|z_s|^2{\rm d}s\\
&\ \ \ \leq |\xi|^p+p\int_{t}^{T}(\mu_s |y_s|^{p-1}\phi^{{1\over p}}(|y_s|^p)+\lambda_s|y_s|^{p-1}|z_s|+|y_s|^{p-1}f_s){\rm d}s\\
&\ \ \ \ \ \ \ -p\int_{t}^{T}|y_s|^{p-2}\mathbbm{1}_{|y_s|\neq0}\left\langle y_{s}, z_{s} {\rm d} B_{s}\right\rangle, \ t\in[0,T],
\end{split}
\end{align}
where $c(p):=\frac{p[(p-1)\wedge1]}{2}$. Note that $\phi(\cdot)$ increases at most linearly since it is a nondecreasing concave function and $\phi(0)=0$. We have
$$\int_0^{T}|y_s|^{p-2}\mathbbm{1}_{|y_s|\neq0}|z_s|^2{\rm d}s<+\infty.$$
It follows from Young's inequality that
\begin{align}\label{2.5}
\begin{split}
p\lambda_s|y_s|^{p-1}|z_s|
&=p\left(\frac{\sqrt{2}\lambda_s}{\sqrt{(p-1)\wedge1}}|y_s|^{\frac{p}{2}}\right)\left(\sqrt{\frac{(p-1)
\wedge1}{2}}|y_s|^{\frac{p-2}{2}}\mathbbm{1}_{|y_s|\neq0}|z_s|\right)\\
&\leq \frac{p\lambda_s^2}{(p-1)\wedge1}|y_s|^p+\frac{c(p)}{2}|y_s|^{p-2}\mathbbm{1}_{|y_s|\neq0}|z_s|^2,\ s\in \T.\vspace{0.2cm}
\end{split}
\end{align}
Combining \eqref{2.4} and \eqref{2.5} yields that for each $t\in[0,T]$,
\begin{align}\label{2.6}
\begin{split}
|y_t|^p+\frac{c(p)}{2}\int_{t}^{T}|y_s|^{p-2}\mathbbm{1}_{|y_s|\neq0}|z_s|^2{\rm d}s\leq X_t-p\int_{t}^{T}|y_s|^{p-2}\mathbbm{1}_{|y_s|\neq0}\left\langle y_{s}, z_{s} {\rm d} B_{s}\right\rangle,
\end{split}
\end{align}
where
$$
X_t:=|\xi|^p+\frac{p}{(p-1)\wedge1}\int_t^T\lambda_s^2|y_s|^p{\rm d}s+p\int_t^T\mu_s|y_s|^{p-1}\phi^{{1\over p}}(|y_s|^p){\rm d}s+p\int_t^T|y_s|^{p-1}f_s{\rm d}s.
$$
In light of $(y_\cdot,z_\cdot)\in \s^p(0,T;\R^k)\times \M^p(0,T;\R^{k\times d})$, it can be easily verified from the BDG inequality that $\{M_t=\int_{0}^{t}|y_s|^{p-2}\mathbbm{1}_{|y_s|\neq0}\left\langle y_{s}, z_{s} {\rm d} B_{s}\right\rangle\}_{t\in[0,T]}$ is a uniformly integrable martingale. By taking the conditional expectation with respect to $\F_t$ in both sides of \eqref{2.6}, we obtain that for each $t\in[0,T]$,
$$
\frac{c(p)}{2}\E\left[\int_t^T{|y_s|^{p-2}}\mathbbm{1}_{|y_s|\neq0}|z_s|^2{\rm d}s\bigg|\F_t\right]\leq\E[X_t|\F_t]
$$
and, in light of the BDG inequality and Young's inequality,
\begin{align*}
\begin{split}
\E\left[\sup_{s\in[t,T]}|y_s|^p\bigg|\F_t\right]
&\leq \E[X_t|\F_t]+p\E\left[\sup_{s\in[t,T]}\left|\int_s^T|y_s|^{p-2}\mathbbm{1}_{|y_s|\neq0}\left\langle y_{s}, z_{s} {\rm d} B_{s}\right\rangle\right|\bigg|\F_t\right]\\
&\leq\E[X_t|\F_t]+\frac{1}{2}\E\left[\sup_{s\in[t,T]}|y_s|^p\bigg|\F_t\right]+\frac{\bar c_p^2}{2}\E\left[\int_t^T{| y_s|^{p-2}}\mathbbm{1}_{|y_s|\neq0}|z_s|^2{\rm d}s\bigg|\F_t\right],
\end{split}
\end{align*}
where $\bar c_p$ is a constant depending only on $p$. Combining the last two inequalities yields the existence of a constant $k_p>0$ depending only on $p$ such that
$$
\E\left[\sup_{s\in[t,T]}|y_s|^p\bigg|\F_t\right]\leq k_p\E[X_t|\F_t],\ \ t\in[0,T].
$$
Furthermore, in light of the last inequality and the definition of $X_\cdot$, by applying \cref{pro:2.3} with $\bar\mu_t:=|y_t|^p$ and Young's inequality we obtain that for each $t\in \T$,
$$
\E\left[\sup_{s\in[t,T]}|y_s|^p\bigg|\F_t\right]
\leq \bar k_pe^{\bar k_pM}\E\left[|\xi|^p+p\int_t^T\mu_s |y_s|^{p-1}\phi^{{1\over p}}(|y_s|^p){\rm d}s+p\int_t^T|y_s|^{p-1}f_s{\rm d}s\bigg|\F_t\right]
$$
and
$$
p\bar k_pe^{\bar k_pM}\E\left[\sup_{s\in[t,T]}|y_s|^{p-1}\int_t^Tf_s{\rm d}s\bigg|\F_t\right]
\leq \frac{1}{2}\E\left[\sup_{s\in[t,T]}|y_s|^p\bigg|\F_t\right]+\frac{k_{p,M}}{2}\E\left[\left(\int_t^Tf_s{\rm d}s\right)^p\bigg|\F_t\right],
$$
where $\bar k_p>0$ is a constant depending only on $p$, and $k_{p,M}$ is another constant depending only on $p$ and $M$. Thus, the desired estimate \eqref{4.23-2} follows immediately from the last two inequalities.\vspace{0.2cm}

Finally, by virtue of Young's inequality and H\"{o}lder's inequality we can deduce that there exists a positive constant $\hat C_{p,M}$ depending only on $p$ and $M$ such that for each $t\in \T$,
\begin{align}\label{4.23-3}
\tilde C_{p, M} \int_{t}^{T} \mu_{s} |y_s|^{p-1}\phi^{{1\over p}}(|y_s|^p) \mathrm{d} s\leq \frac{1}{2}\sup _{s \in[t, T]}\left|y_{s}\right|^{p} +\hat C_{p, M}\left(\int_{t}^{T}\mu_{s}\phi^{{1\over p}}(|y_s|^p) \mathrm{d}s\right)^p
\end{align}
and
\begin{align}\label{4.23-4}
\left(\int_{t}^{T}\mu_{s}\phi^{{1\over p}}(|y_s|^p) \mathrm{d}s\right)^p\leq
\left(\int_{t}^{T}\mu_{s}\mathrm{d}s\right)^{p-1}
\int_{t}^{T}\mu_{s}\phi(|y_s|^p) \mathrm{d}s\leq M^{p-1} \int_{t}^{T}\mu_{s}\phi(|y_s|^p) \mathrm{d}s.
\end{align}
Then, the desired assertion in \cref{pro:4.7} follows from inequalities \eqref{4.23-1}, \eqref{4.23-2}, \eqref{4.23-3} and \eqref{4.23-4}.\vspace{0.1cm}

\section*{Acknowledgements}
The authors are greatly grateful to the referee for his/her careful reading
and constructive \vspace{-0.2cm}suggestions.




\begin{thebibliography}{21}

\setlength{\baselineskip}{13.5pt}
\setlength{\itemsep}{1.8mm}


\bibitem[Briand et al.(2003)]{Bri03}Briand Ph., Delyon B., Hu Y., Pardoux E., Stoica L., 2003. $L^p$ solutions of backward stochastic differential equations. {\it Stochastic Processes and Their Applications} {\bf 108}, 109-129.

\bibitem[Briand and Hu(2006)]{BriHu06}Briand Ph., Hu Y., 2006. BSDE with quadratic growth and unbounded terminal value. {\it Probability Theory and Related Fields} {\bf 136}, 604-618.

\bibitem[Briand and Hu(2008)]{BriHu08}Briand Ph., Hu Y., 2008. {Q}uadratic {BSDE}s with convex generators and unbounded terminal conditions. {\it Probability Theory and Related Fields} {\bf 141}(3), 543-567.

\bibitem[Chen and Wang(2000)]{ChenWang00}Chen Z., Wang B., 2000. Infinite time interval BSDEs and the convergence of $g$-martingales. {\it Journal of the Australian Mathematical Society (Series A)} {\bf 69}, 187-211.

\bibitem[Delbaen and Tang(2010)]{Delb10}Delbaen F., Tang S., 2010. Harmonic analysis of stochastic equations and backward stochastic differential equations. {\it Probability Theory and Related Fields} {\bf 146}, 291-336.

\bibitem[El Karoui, Peng and Quenez(1997)]{El97}El Karoui N., Peng S., Quenez M.-C., 1997. Backward stochastic differential equations in finance. {\it Mathematical Finance} {\bf 7}, 1-71.

\bibitem[Fan(2015)]{Fan15}Fan S., 2015. $L^p$ solutions of multidimensional BSDEs with weak monotonicity and general growth generators. {\it Journal of Mathematical Analysis and Applications} {\bf 432}, 156-178.

\bibitem[Fan(2016)]{Fan16}Fan S., 2016. Bounded solution, $L^p~(p>1)$~solutions and $L^1$ solutions for one-dimensional BSDEs under general assumptions. {\it Stochastic processes and their applications} {\bf 126}, 1511-1552.

\bibitem[Fan(2018)]{Fan18}Fan S., 2018. Existence, uniqueness and stability of $L^1$ solutions for multidimensional BSDEs with generators of one-sided Osgood type. {\it Journal of Theoretical Probability} {\bf 31}, 1860-1899.

\bibitem[Fan and Jiang(2010)]{FanJiang10}Fan S., Jiang L., 2010. Uniqueness result for the BSDE whose generator is monotonic in $y$ and uniformly continuous in $z$. {\it C. R. Acad. Sci. Paris, Ser. I} {\bf 348}, 89-92.

\bibitem[Fan and Jiang(2013)]{FanJiang13}Fan S., Jiang L., 2013. Multidimensional BSDEs with weak monotonicity and general growth generators. {\it Acta Mathematica Sinica (English Series)} {\bf 23}(10), 1885-1906.

\bibitem[Fan and Jiang(2019)]{FanJiang19}Fan S., Jiang L., 2019.
$L^p$ solutions of BSDEs with a new kind of non-Lipschitz coefficients. {\it Acta Mathematicae Applicatae Sinica, English Series} {\bf 35}(4), 695-707.

\bibitem[Fan, Hu and Tang(2023a)]{FanHuTang2023SCL}
Fan S., Hu Y., Tang S., 2023. $L^1$ solution to scalar BSDEs with logarithmic sub-linear growth generators. {\it System and Control Letters} {\bf 177}, 105553.

\bibitem[Fan, Hu and Tang(2023b)]{FanHuTang2023JDE}
Fan S., Hu Y., Tang S., 2023. Multi-dimensional backward stochastic differential equations of diagonally quadratic generators: The general result. {\it Journal of Differential Equations} {\bf 368}, 105-140.

\bibitem[Fan, Hu and Tang(2024)]{FanHuTang2024SCL2}
Fan S., Hu Y., Tang S., 2024. Scalar BSDEs of iterated-logarithmically sub-linear generators with integrable terminal values. {\it System and Control Letters} {\bf 188}, 105805.

\bibitem[Hamad\`{e}ne(2003)]{Ham03}Hamad\`{e}ne S., 2003. Multidimensional backward stochastic differential equations with uniformly continuous coefficients. {\it Bernoulli} {\bf 9}(3), 517-534.

\bibitem[Hu, Nualart and Song(2011)]{Hu11}Hu Y., Nualart D., Song X., 2011. Malliavin calculus for backward stochastic differential equations and application to numerical solutions. {\it The Annals of Applied Probability} {\bf 21}(6), 2379-2423.

\bibitem[Hu and Tang(2016)]{HuTang2016SPA}Hu Y., Tang S., 2016. Multi-dimensional backward stochastic differential equations of diagonally quadratic generators. {\it Stochastic Processes and Their Applications} {\bf 126}(4), 1066-1086.

\bibitem[Jia(2010)]{Jia10}Jia G., 2010. Backward Stochastic differential equations with a uniformly continuous generator and related $g$-expectation. {\it Stochastic Processes and Their Applications} {\bf 120}(11), 2241-2257.

\bibitem[Klimsiak and Rzymowski(2024)]{Klimsiak2024ECP}Klimsiak T., Rzymowski M., 2024. A priori estimates for multidimensional BSDEs with integrable data. {\it Electronic Communications in Probability} {\bf 29}(34), 1-12.

\bibitem[Kobylanski(2000)]{Kob00}Kobylanski M., 2000. Backward stochastic differential equations and partial equations with quadratic growth. {\it The Annals of Probability} {\bf 28}, 259-276.

\bibitem[Li, Xu and Fan(2021)]{LiXuFan2021}Li T., Xu Z., Fan S., 2021. General time interval multidimensional BSDEs with generators satisfying a weak stochastic-monotonicity condition. {\it Probability, Uncertainty and Quantitative Risk} {\bf 6}(4), 301-318.

\bibitem[Liu, Li and Fan(2020)]{LiuLiFan2020}Liu Y., Li D., Fan S., 2020. $L^p~(p>1)$ solutions of BSDEs with generators satisfying some non-uniform conditions in $t$ and $\omega$. {\it China Annals of Mathematics, Series B} {\bf 41}(3), 479-494.

\bibitem[Li and Fan(2023)]{LiFan2023}Li X., Fan S., 2023. $L^p$ solutions of general time interval BSDEs with generators satisfying a $p$-order weak stochastic-monotonicity condition. {\it Communications in Statistics-Theory and Methods} {\bf 52}(16), 5650-5676.

\bibitem[Li, Lai and Fan(2023)]{LiLaiFan2023}Li X., Lai Y., Fan S., 2023. BSDEs with stochastic Lipschitz condition: A general result. {\it Probability, Uncertainty and Quantitative Risk} {\bf 8}(2), 267-280.

\bibitem[Mao(1995)]{Mao1995}Mao X., 1995. Adapted solutions of backward stochastic differential equations with non-Lipschitz cofficients, {\it Stochastic Processes and Their Applications} {\bf 58}, 281-292.

\bibitem[Pardoux and Peng(1990)]{PardouxPeng1990}Pardoux E., Peng S., 1990. Adapted solution of a backward stochastic differential equation, {\it Systems and Control Letters} {\bf 14}, 55-61.

\bibitem[Peng(1997)]{Peng1997}Peng S., 1997. Backward SDE and related $g$-expectation. In: El Karoui, N., Mazliak, L. (Eds.), {\it Backward Stochastic Differential Equations}, Pitman Research Notes: Mathematical Series, Vol. 364 Longman, Harlow (1997) pp.141-159.

\bibitem[Xiao, Fan and Xu(2015)]{XiaoFanXu15}Xiao L., Fan S., Xu N., 2015. $L^p~(p\geq1)$ solutions of multidimensional BSDEs with monotone generators in general time intervals. {\it Stochastics and Dynamics} {\bf 1}, 1-34.

\bibitem[Xiao and Fan(2020)]{XiaoFan20}Xiao L., Fan S., 2020. $L^p~(p\geq1)$ solutions of multidimensional BSDEs with time-varing quasi-H\"{o}lder continuity generators in general time intervals. {\it Communications of the Korean Mathematical Society} {\bf 35}(2), 667-684.

\bibitem[Zong and Hu(2018)]{ZongHu18}Zong G., Hu F., 2018. $L^p$ solutions of infinite time interval backward doubly stochastic differential equations under monotonicity and general increasing conditions. {\it Journal of Mathematical Analysis and Applications} {\bf 458}(2), 1486-1511.

\end{thebibliography}
\end{document}